\documentclass[a4paper,11pt]{amsart}

\usepackage[a4paper, left=2.3cm, right=2.3cm, top=2cm, bottom=2cm]{geometry}
\usepackage{meintex}
\usepackage{amsmath}
\usepackage{amsthm}
\usepackage{amssymb}
\usepackage{todonotes}
\usepackage{enumitem}
\usepackage[section]{placeins}

\newcommand{\N}{\mathbb{N}}

\newcommand{\calP}{\mathcal P}
\newcommand{\calC}{\mathcal C}
\newcommand{\calI}{\mathcal I}
\newcommand{\calG}{\mathcal G}
\newcommand{\calA}{\mathcal A}
\newcommand{\calK}{\mathcal K}

\newcommand{\h}{\textbf{h}}
\newcommand{\R}{\mathbb R}
\newcommand{\Uad}{\mathcal U_{\text{ad}}}
\newcommand{\Lip}{\text{Lip}}

\title[Mean-field optimal control]{Mean-field optimal control and optimality conditions in the space of probability measures}

\author[M.~Burger]{Martin Burger}
\address[M.~Burger]{Department of Mathematics, Friedrich-Alexander Universit\"at  Erlangen-N\"urnberg, Cauerstra{\ss}e 11, 91058 Erlangen}
\email{martin.burger@fau.de}

\author[R.~Pinnau]{Ren\'e Pinnau}
\address[R.~Pinnau]{Department of Mathematics, TU Kaiserslautern \\ Erwin-Schr\"odinger-Str.\ 48, 67663 Kaiserslautern}
\email{pinnau@mathematik.uni-kl.de}

\author[C.~Totzeck]{Claudia Totzeck}
\address[C.~Totzeck]{Department of Mathematics, TU Kaiserslautern \\ Erwin-Schr\"odinger-Str.\ 48, 67663 Kaiserslautern}
\email{totzeck@mathematik.uni-kl.de}

\author[O.~Tse]{Oliver Tse}
\address[O.~Tse]{Department of Mathematics and Computer Science, Eindhoven University of Technology \\ P.O. Box 513, 5600MB Eindhoven, The Netherlands}
\email{o.t.c.tse@tue.nl}

\begin{document}

\begin{abstract}
We derive a framework to compute optimal controls for problems with states in the space of probability measures. Since many optimal control problems constrained by a system of ordinary differential equations (ODE) modelling interacting particles converge to optimal control problems constrained by a partial differential equation (PDE) in the mean-field limit, it is interesting to have a calculus directly on the mesoscopic level of probability measures which allows us to derive the corresponding first-order optimality system. 
In addition to this new calculus, we provide relations for the resulting system to the first-order optimality system derived on the particle level, and the first-order optimality system based on $L^2$-calculus under additional regularity assumptions. We further justify the use of the $L^2$-adjoint in numerical simulations by establishing a link between the adjoint in the space of probability measures and the adjoint corresponding to $L^2$-calculus.
Moreover, we prove a convergence rate for the convergence of the optimal controls corresponding to the particle formulation to the optimal controls of the mean-field problem as the number of particles tends to infinity.

\end{abstract}

\maketitle


\medskip


\section{Introduction}

In the past few years, the growing interest in the (optimal) control of interacting particle systems and their corresponding mean-field limits led to many contributions on their numerical behavior (see, e.g.,~\cite{diss, sheep1}) as well as their analytical properties, e.g.,\cite{FornasierSolombrino,FornasierPP}. They can be found in various fields of applications, for example physical or biological models like crowd dynamics \cite{sheep1,Borzi,zuazua,bongini}, consensus formation  \cite{Dante}, or even global optimization \cite{CBO1,CBO2}. Meanwhile, there are also first approaches for stochastic particle systems available \cite{bonnet,pham}.

Since there are several points of view on this subject, the analytical techniques vary from standard ODE and PDE theory over optimal transport to measure-valued solutions. This induces also different variants for the derivation of first-order optimality conditions and/or  gradient information, which clearly also has some impact on the design of appropriate numerical algorithms for the solution of the optimal control problems at hand.

Before we discuss the novelty and advantages of our approach we recall some recent contributions to the topic. In \cite{FornasierSolombrino} the notion of mean-field optimal control problems was introduced. The authors combine well-known mean-field limit results with $\Gamma$-convergence to prove the convergence of optimal controls of the microscopic problem with $N$ interacting particles to a solution of the corresponding mean-field optimal controls. The article focuses on sparse controls and Caratheodory solutions, where the controls act linearly and additive on the dynamic of the interacting particles. In contrast to the present paper, there is no discussion of first order optimality conditions, no statement of adjoints, and no discussion of a convergence rate.

Based on these observations, the derivation of a mean-field Pontryagin maximum principle was shown in \cite{FornasierPP}. Starting from a Hamiltonian point of view, subdifferential calculus is employed to derive a gradient flow structure and a corresponding forward-backward system. Key ingredients of the proofs are semiconvexity of the functionals along geodesics and a rescaling of the adjoint variable. The article considers a dynamical system of interacting particles and additionally some policy makers. The controls enter the dynamics through the policy makers which remain finite as the number of interacting particles tends to infinity. To illustrate their methodology, explicit computations for the Cucker-Smale dynamics were presented. As in  \cite{FornasierSolombrino}, discussions on the optimality conditions, adjoints or convergence rate were absent.

Another Pontryagin Maximum Principle was derived via subdifferential calculus on the space of probability measures and needle-like variations in \cite{bonnet} for a non-local transport equation, where the control variable enters linearly in the velocity field of the transport equation. The resulting first-order optimality system consists of a forward-backward equation similar to the one in \cite{FornasierPP}, and the corresponding measure is identified by disintegration. Supplementary to the needle-like variations, we propose a different approach for the derivation of a corresponding linear system, and as a by-product, provide a direct link between the particle adjoint and its mean-field counterpart. Furthermore, the velocity fields considered in the present paper are more general. In addition, we provide a convergence rate as the number of interacting particles tends to infinity.

In contrast to these analytical results, \cite{herty18} approaches the problem formally with techniques from the field of Optimization with PDE constraints. All assumptions and computations are formal and the mean-field limit is established via an BBGKY-approach. Adjoints are dervided with formal $L^2$-calculus and closed by moments which can be interpreted as conditional expectations. 
A similar formal derivation can be found in \cite{mfchierachy}.

\medskip

To summarize, the aim of our contribution is multi-fold: 
\begin{enumerate}
    \item We take an applied viewpoint and establish first-order optimality conditions, in the KKT-sense, on the space of probability measures via a Lagrangian approach which can be used for numerical implementations. While the derivation of the linearized system (see \eqref{eq:linearization} in Lemma~\ref{lem:4.4}) bears similarities with those made in \cite{bonnet}, we provide an alternative strategy that circumvents the explicit use of Lagrangian flows. Additionally, we provide a characterization of the corresponding adjoint system, which takes the form of a momentum equation (see Theorem~\ref{eq:opt_W}). As the considerations can be lifted in a straightforward manner to second-order dynamical systems, we rigorously justify the numerical results shown in \cite{sheep1}.
    \item We build the bridge between the Hamiltonian-based results discussed in \cite{FornasierSolombrino, FornasierPP, mfchierachy, bonnet} and the ones obtained by Lagrangian approaches (see the chart in Section \ref{sec:relations}).
    \item We prove the convergence, with rates, of the sequence of optimal controls, as the number of interacting particles tends to infinity (cf.\ Theorem~\ref{thm:convergenceRate}). The convergence crucially relies on the optimality system obtained in (1).
\end{enumerate}
 
The main ideas are discussed in the following model example before we present our results in full details. 

\subsection{An illustrative example: Controlling a single particle}

Let us start with an illustrative example from classical optimal control in order to illustrate the idea without the complication of a mean-field limit. We denote the dimension of the state space by $d \ge 1$ and the time interval of interest is $[0,T]$ for some $T >0$.  We assume that the control variable $u$ acts on the velocity of a single particle with trajectory $x_t \in \R^d$ for $t \in [0,T]$ and we want to optimize a given functional depending on  the trajectory, i.e.\
\begin{equation}
(x,u) = \argmin \int_0^T g(x_t)\dd t, \quad \text{subject to}\qquad  \frac{\dd}{\dd t} x_t = v(x_t, u_t), 
\end{equation}
where $g$ and $v$ are given, sufficiently regular functions.

Then, the standard Pontryagin Maximum principle yields the existence of an adjoint variable $\xi$ satisfying 
\begin{equation}
\frac{\dd}{\dd t} \xi_t = \nabla_x g(x_t) + \nabla_x v(x_t,u_t) \xi_t,
\end{equation}
with terminal condition $\xi_T =0$. 

Moreover, the control $u$ satisfies the optimality condition 
\[
\nabla_u v(x_t,u_t) \cdot \xi_t = 0  \quad \text{ a.e. in } (0,T).
\]
These conditions can be translated into the calculation of a saddle-point of the microscopic Lagrangian
\begin{equation}
{\mathcal L}_{\text{micro}}(x,u,\xi) = \int_0^T g(x_t) \dd t + \int_0^T \left(\frac{\dd}{\dd t} x_t - v(x_t, u_t)\right) \cdot \xi_t \dd t.
\end{equation}
On the other hand, the discrete ODE can be translated into a macroscopic formulation via the method of characteristics: with initial value $\mu_0=\delta_{x_0}$ the concentrated measure
$\mu_t = \delta_{x_t}$ is the unique solution of 
\begin{equation}
\partial_t \mu_t + \nabla_x \cdot (v(x_t,u_t) \mu_t) = 0. 
\end{equation}
Since all measures are concentrated at $x_t$ we can reinterpret $u_t$ as the evaluation of a feedback-control $u(x,t)$ at $x=x_t$ and equivalently obtain
\begin{equation} \label{eq:transporteq0}
\partial_t \mu_t + \nabla_x \cdot (v(x,u_t) \mu_t) = 0, \qquad  \mu_0=\delta_{x(0)}.
\end{equation}
Since
$$ \int_0^T  g(x_t) \dd t  = \int_0^T \langle g, \mu_t\rangle \dd t,  $$
we can formulate an optimal control problem at the macroscopic level for the measure $\mu$ and the control variable $u$, i.e.\ 
\begin{equation}
(\mu, u) = \argmin \int_0^T \langle g, \mu_t\rangle \dd t  \qquad \text{subject to \eqref{eq:transporteq0}}.   
\end{equation} 
This macroscopic optimal control problem is in fact  equivalent to the microscopic one for a single particle, since we can choose the state space as the Banach space of Radon measures and the control space as an appropriate space of reasonably smooth functions on $\R^d \times (0,T)$. The uniqueness of solutions to the transport equation and the special initial value will always yield a concentrated measure and the identification $u_t = u(x_t,t)$ brings us back to the microscopic control.

However, with the macroscopic formulation we have another option to derive optimality conditions in these larger spaces, based on the Lagrangian 
\begin{equation}
{\mathcal L}_{\text{macro}}(\mu,u,\varphi) = \int_0^T \langle g, \mu_t\rangle \dd t + \int_0^T \langle \varphi,  	\partial_t \mu_t + \nabla_x \cdot (v(x,u_t) \mu_t) \rangle \dd t.
\end{equation}
Then, the macroscopic adjoint equation becomes 
\begin{equation}
\partial_t \varphi + v(x, u_t) \cdot \nabla_x \varphi = 0 
\end{equation}
and the optimality condition is given by
$$ - \langle \nabla_x \varphi, \nabla_u v(x,u_t) \mu_t \rangle = 0. $$
Due to the equivalence of the microscopic and macroscopic optimal control problem it is natural to ask for the relation between the adjoint variables $\xi$ and $\varphi$, which is not obvious at a first glance and yet only very little discussed. For first results in this direction see \cite{herty18}. Using the special structure of the solution $\mu_t$ and the identification with the microscopic control we can rewrite the optimality condition as
$$ \nabla_u v(x_t,u_t) \cdot (- \nabla_x \varphi(x_t,t)) = 0, $$
which induces the identification
\begin{equation}
\xi_t = - \nabla_x \varphi(x_t,t).
\end{equation}
Indeed, the method of characteristics confirms that $- \nabla_x \varphi(x_t,t)$ satisfies the microscopic adjoint equation.
This becomes more apparent if we consider only variations of $\mu$ that respect the nonnegativity and mass one condition of the probability measure, i.e.\ 
$$ \mu'  = - \nabla \cdot q, $$
with a vector-valued measure $q$ being absolutely continuous with respect to $\mu$. Then, an integration by parts argument directly reveals the relation to $-\nabla \varphi$.

By using variations of this kind we reinterpret the state space as a Riemannian manifold of Borel probability measures equipped with the 2-Wasserstein distance instead of the flat Banach space of Radon measures. The analysis of particle systems and limiting nonlinear partial differential equations in the 2-Wasserstein distance has been a quite fruitful field of study in the last years following the seminal papers \cite{Otto,JKO}. It is hence highly overdue to study such an approach also in the optimal control setting. 

We mention that the values of $\varphi$ outside the trajectory  are irrelevant for the specific control problem. Solving 
$$ \partial_t \varphi + v(\cdot, u_t) \cdot \nabla_x \varphi = 0,\qquad \nabla_u v(\cdot,u_t) \cdot \nabla \varphi = 0 \quad \text{on } \R^d \times (0,T), $$
we obtain the adjoints for all possible microscopic control problems with initial value in $\R^d$. This is just the well-known Hamilton-Jacobi-Bellmann equation, usually derived with different arguments.

\begin{rem}
	The above arguments can also be extended to a stochastic control system (see, e.g., \cite{roy2018}): 
	\begin{equation}
	(X,u) = \argmin \int_0^T E^x[g(X_t)] \dd t, \quad \text{subject to}\quad  {\dd}  X_t = v(X_t,u_t) \dd t + \sigma (X_t,u_t) \dd W_t, 
	\end{equation}
	with $W_t$ being a Wiener process and $X$ the solution to the stochastic differential equation with initial condition $X_0=x$. In this case the state equation for the probability density $\mu$ becomes 
	\begin{equation}
	\partial_t \mu_t + \nabla \cdot (v(x,u_t) \mu_t) = \frac{1}2 \Delta( \sigma^2 \mu_t), 
	\end{equation}
	and $\mu$ does not necessarily remain a concentrated measure in time, which corresponds to the stochasticity of the model. 
\end{rem}

\subsection{Control in the Mean-field Limit}

Having understood the relation between microscopic and macroscopic formulations of the optimal control problem, it seems an obvious step to consider optimal control problems for a high number of particles $N$ and their mean-field limit as $N\rightarrow \infty$, which is also the motivation for this paper. However, in the mean-field limit there is no microscopic particle system and corresponding optimal control problem, hence an additional step is needed to understand the connection in the limit. The basis for such a step is to understand the characteristic flow, which replaces the particle dynamics and naturally leads to an analysis in the Wasserstein distance. We will further investigate this mean-field setting in the remainder of the paper.

Here, we restrict our considerations to first order dynamics, but the present paper can be seen as an analytical justification of the convergence shown numerically in \cite{sheep1}. It is an additional contribution to the field of optimization of particle systems and their mean-field limits which is lively discussed in the recent years (e.g.\ \cite{Borzi,Dante,FornasierSolombrino,FornasierPP,CBO1,CBO2,Giacomo,mfchierachy}). Moreover, we would like to connect the fields of optimal control and gradient flows as well as optimal transport. In particular, we show relations between the adjoints derived by $L^2$-calculus and adjoints derived in the space of probability measures ($W_2$-adjoints).

The paper is organized as follows: in Section~\ref{sec:state} the microscopic model for $N$ particles and the corresponding mean-field equation is introduced. Further, we formulate the optimal control problems under investigation. The first main contribution of the article is the derivation of the first-order optimality conditions in the mesoscopic formulation given in Section~\ref{sec:FONC}. A discussion of the relation of this new calculus to the first-order optimality systems on the particle level and the first-order optimality condition based on $L^2$-calculus is the content of Section~\ref{sec:relations}. In Section~\ref{sec:ConvergenceRate} we show the second main result which is the convergence rate for the optimal controls as $N\to \infty$.

\section{Optimal Control Problems}\label{sec:state}
\noindent
First, we generalize the one-particle case  to  $N\in\N$ interacting particles, modeling, e.g., crowd dynamics \cite{sheep1}. Then, we derive its corresponding mean-field limit, i.e.\ the mesoscopic approximation. These two are the state systems for the respective optimal control problems. Further, we present the assumptions which are necessary for the well-posedness of the state systems.

\subsection{The State Models}
As before, $d\ge 1$ denotes the dimension of the state space and $[0,T]\subset\R$ with $T > 0$ is the time interval of interest. 
\subsubsection{The particle system}
The considered particle system consists of $N \in \mathbb{N}$ particles of the same type and $M\in\N$ controls represented by the functions
\begin{equation*}
x^i, u^\ell \colon [0,T] \rightarrow \mathbb{R}^d, \qquad \text{for } i = 1,\dots N \text{\;\;and\;\;}\ell=1,\dots,M.
\end{equation*}
The vectors 
\begin{equation*}
\x := (x^i)_{i=1,\dots,N}, \qquad \con := (u^\ell)_{\ell=1,\dots,M},
\end{equation*}
denote the states of the particles and the controls, respectively.

The particle system reads explicitly
\begin{equation}\label{eq:stateODE}
\frac{\dd}{\dd t} \x_t = v^N(\x_t,\con_t), \qquad \x_0=\hat \x,
\end{equation}
with given  $\hat \x \in \R^{dN}$ defining the initial states of the particles. 
The operator $v^N$ on the right-hand side strongly depends on the type of application. 

In the following, we denote by $\calP_2(\R^d)$ the space of Borel probability measures on $\R^d$ with finite second moment and equipped with the 2-Wasserstein distance, which makes $\calP_2(\R^d)$ a complete metric space, and by $\calP_2^{ac}(\R^d)$ the subset of $\calP_2(\R^d)$ containing probability measures with Lebesgue density. For the sake of completeness we recall the 2-Wasserstein distance:
\[
W_2^2(\mu, \nu) := \inf\limits_{\pi \in \Pi(\mu, \nu)}\biggl\{\int_{\R^d} |x -y|^2 \dd\pi(x,y)\biggr\},\qquad \mu,\nu\in \calP_2(\R^d),
\]
where $\Pi(\mu,\nu)$ denotes the set of all Borel probabililty measures on $[0,T] \times \R^{2d}$ that have $\mu$ and $\nu$ as first and second marginals respectively, i.e.\
\[
\pi(B \times \R^d) = \mu(B), \qquad \pi(\R^d \times B) = \nu(B) \quad \text{for } B \in \mathcal{B}(\R^{d}).
\]
In the rest of the article we denote by $\mathfrak{m}_2(\mu)$ the second moment of $\mu\in\calP_2(\R^d)$.

We further assume

\medskip

\begin{enumerate}[label=(A\arabic*)]
	\setcounter{enumi}{0}
	\item \label{A1} Let $v: \calP_2(\R^d) \times \R^{dM} \to \Lip_{loc}(\R^d)$ be given, such that for all $(\mu,\con) \in  \calP_2(\R^d) \times \R^{dM}$:
	\[
	 \langle v(\mu,\con)(x)-v(\mu,\con)(y),x-y\rangle \le C_l|x-y|^2,\qquad x,y\in\R^d.
	\]
	where the constant $C_l>0$ is independent of $(\mu,\con)$.
	
	We further define $v^N\colon \R^{dN}\times \R^{dM} \to \R^{dN}$ via
	$$
	 v_i^N(\x,\con) := v(\mu^N,\con)(x^i), \quad i=1,\ldots, N,
	$$
	where
	\begin{equation*}
		\mu_{\x}^N(A) = \frac{1}{N} \sum_{i=1}^N \delta_{x^i}(A),\qquad A\in\mathcal{B}(\R^d)\;(=\text{Borel $\sigma$-algebra}),
	\end{equation*}
	is the empirical measure for the state $\x\in\R^{dN}$.

	\item \label{A2}For any two $(\mu,\con), (\mu',\con')\in \calP_2(\R^d) \times \R^{dM}$, there exists a constant $C_v>0$, independent of $(\mu,\con)$ and $(\mu',\con')$, such that
	\[
	 \|v(\mu,\con)-v(\mu',\con')\|_{\sup} \le C_v\Bigl( W_2(\mu,\mu') + \|\con-\con'\|_2\Bigr).
	\]

\end{enumerate}

\begin{rem}
	By definition, $\mu_t^N$ assigns the probability $\mu_t^N(A)$ of finding particles with states within a measurable set $A \in \mathcal{B}(\R^d)$ on the state space $\R^d$ at time $t\ge0$. 
\end{rem}

Standard results from ODE theory yield the existence and uniqueness of a global solution.
\begin{prop}
	Assume \ref{A1} and \ref{A2}. Then, for given $\con \in \calC([0,T],\R^{dM})$ and $\hat \x \in \R^{dN}$ there exists a unique global solution $\x \in \calC^1([0,T],\R^{dN})$ of (\ref{eq:stateODE}).
\end{prop}

\begin{rem}
	In particular, for applications in the control of crowds we have that $v^N$ models interactions, i.e.\ particle-particle and particle-control interactions by means of forces (see \cite{CarrilloSurvey} and the references therein). Then, $v^N$  is often given by
	\begin{equation}\label{eq:structureV}
		v_i^N(\x, \con) = -\frac{1}{N}\sum_{j=1}^N K_1(x^i - x^j) - \sum_{\ell=1}^M K_2(x^i - u^\ell),
	\end{equation}
	for given interaction forces $K_1$ and $K_2$ modeling the interactions within the cloud of particles itself and of the particles with the controls, respectively.
\end{rem}

\subsubsection{The mean-field model}

In order to define the limiting problem for an increasing number of particles $N\to\infty$ explicitly, we consider the empirical measure
$\mu^N$.

Using the ideas from \cite{Neunzert,BraunHepp,Dobrushin} we derive the corresponding PDE formally as
\begin{equation}\label{longPDE}
\partial_t \mu_t + \nabla\cdot \big( v(\mu_t,\con_t) \mu_t \big)=0, \qquad \mu_0 = \hat \mu,
\end{equation}
which is the mean-field 1-particle distribution evolution equation, supplemented with the initial condition $\hat \mu \in  \calP_2^{ac}(\R^d)$, i.e.\ $\hat\mu$ has Lebesgue density.

\begin{rem}
	Here $v(\mu,\con)$ denotes the mean-field representation of $v^N(\x,\con)$. In fact, for the structure given by \eqref{eq:structureV}, we obtain
	\begin{equation}\label{mfV}
	(t,x)\mapsto v(\mu_t,\con_t)(x) = -(K_1 \ast \mu_t)(x) - \sum_{\ell=1}^M  K_2 (x - u_t^\ell).
	\end{equation}
\end{rem}

In the mean-field setting we consider the following notion of solution.
\begin{defi}
	We call $\mu\in \calC([0,T],\calP_2(\R^d))$ a {\em weak measure solution} of \eqref{longPDE} with initial condition $\hat\mu\in \calP_2(\R^d)$ iff for any test function $h\in \mathcal{C}_0^\infty([0, T)\times \R^d)$ we have
	\[
	\int_0^T \int_{\R^d} \big( \partial_t h_t  + v(\mu_t,\con_t)\cdot\nabla h_t \big)\dd \mu_t\dd t + \int_{\R^d} h_0\dd \hat\mu = 0.
	\]
\end{defi}

An existence and uniqueness result for solutions of \eqref{longPDE} may be found, e.g., in \cite{BraunHepp,WassersteinConvergence,Dobrushin,Golse}, where the notion of solution is established in the Wasserstein space $\calP_2(\R^d)$:
\begin{prop}\label{prop:exist_cont}
	Assume \ref{A1} and \ref{A2} and let $\hat \mu \in \calP_2(\R^d)$. Then, for  $\con \in \calC([0,T],\R^{dM})$ there exists a unique global (weak measure) solution $\mu \in \calC([0,T],\calP_2(\R^d))$ of (\ref{longPDE}). If additionally $\hat\mu\in\calP_2^{ac}(\R^d)$, then also $\mu\in \calC([0,T],\calP_2^{ac}(\R^d))$.
	
	\
	
	Further, for $\hat \mu_{\hat \x} =  1/N \, \sum_{i=1}^N \delta_{\hat x^i}$ we have $\mu_{\x,t} = \mu_{\x,t}^N$, where $\hat \x $ is the initial condition of (\ref{eq:stateODE}).
\end{prop}

\begin{rem}
Under the assumptions \ref{A1} and \ref{A2} we have enough regularity to use the classical method of characteristics to deduce for any $s \in [0,T]$ the existence of an unique global flow $Q_{\cdot}(\cdot,s) \in \mathcal{C}([0,T] \times \R^d; \R^d)$ satisfying
\begin{equation}\label{eq:W-Flow}
\frac{d}{dt} Q_t(x,s) = v(\mu_t,\con_t) \circ Q_t(x,s), \qquad Q_s(x,s) = x.
\end{equation}
In particular, for $s=0$ we obtain the nonlinear flow with a random initial condition $Q_0(x,0)$ distributed according to $\hat\mu$, i.e.\ $\text{law}(Q_0(x,0))=\hat \mu$. The solution $\mu$ of \eqref{longPDE} may then be explicitly expressed as $\mu_t = Q_t(\cdot,0)\#\mu_0$ for all $t \ge 0.$ We shall make use of this representation at several points in the remainder. For simplicity we set $Q_t(x) := Q_t(x,0).$
\end{rem}

The following stability statement will be useful in the coming results. Its proof may be found in Appendix~\ref{appex:forwardDobrushin}. 

\begin{lem}\label{lem:forwardDobrushin}
 Let the assumptions \ref{A1} and \ref{A2} hold, $\mu$ and $\mu'$ be solutions to the continuity equation \eqref{eq:PDE} for given controls $\con$, $\con'$ and initial data $\hat \mu$, $\hat\mu'$, respectively. Then, there exist positive constants $a$ and $b$ such that
 \begin{equation*}
 W_2^2(\mu_t,\mu_t') \le \Bigl(W_2^2(\hat \mu,\hat \mu') + b\|\con-\con'\|_{L^2((0,T),\R^{dM})}^2\Bigr) e^{a t}\qquad \text{for all\, $t\in[0,T]$}.
 \end{equation*}
\end{lem}

We end this section with an important observation:

\begin{rem}\label{generalization}
We emphasize that the particle problem is just a special case of the mean-field problem specified by the inital condition. Indeed, for the initial condition $\hat \mu =  1/N \, \sum_{i=1}^N \delta_{\hat x^i}$ we have $\mu_t = \mu_t^N$, where $\hat \x $ is the initial condition of (\ref{eq:stateODE}). Strictly speaking, we have only one optimization problem to consider in the following. Whether the problem at hand is of microscopic or mesoscopic type is determined by the initial condition. 
\end{rem}

\subsection{Optimal Control Problem}\label{sec:OptProb}

We define the set of admissible controls as
\begin{equation}\label{eq:Uad}
\Uad = \{ \con \in H^1((0,T), \R^{dM}) \colon \con_0 = \hat \con\}, \quad \text{with}\;\; \hat\con \in \R^{dM} \; \text{given}. 
\end{equation}
This choice of $\Uad$ ensures the continuity of the controls (compare also the previous existence results).

\medskip

For the study of the respective optimal control problem we require:

\begin{enumerate}[label=(A\arabic*)]
	\setcounter{enumi}{2}
	\item \label{A3} The cost functional is of separable type, i.e.\ 
	\begin{equation}\label{eq:cost}
	    J(\mu,\con) = \int_0^T J_1(\mu_t)\dd t + J_2(\con),
	\end{equation}
	where $J_2$ is continuously differentiable, weakly lower semicontinuous and coercive on $\Uad$. Further, $J_1(\mu)$ is a cylindrical function of the form
	\[
	    J_1(\mu) = j(\langle g_1,\mu \rangle, \dots, \langle g_L, \mu \rangle),
	\]
	where $j \in \mathcal{C}^1(\R^L)$ and $g_\ell \in \mathcal{C}^1(\R^d),\; \ell=1,\dots,L,$ such that $\langle g_\ell,\mu\rangle:=\int_{\R^d} g_\ell\dd\mu <\infty$, and
	\[
	 |\nabla g_\ell|(x) \le C_g(1+|x|)\qquad \text{for all $x\in\R^d$ and $\ell=1,\ldots,L$},
	\]
	for some constant $C_g>0$.
	
	\item \label{A4} For the microscopic case, we define $J_1^N(\x) := J_1(\mu_{\x}^N)$ as well as
	\begin{equation}\label{eq:cost_N}
	    J_N(\x,\con) := \int_0^T J_1^N(\x_t)\dd t +J_2(\con),
	\end{equation}
	and assume that $J_1^N $ is continuously differentiable.
\end{enumerate}

\begin{rem}
	\begin{enumerate}
        \item Note, that the differentiability properties in the previous assumptions are only necessary for the derivation of the optimality conditions in the next sections, and not for the existence of the respective optimal controls.
        \item (A4) essentially restricts the type of costs that can be considered for the particle system. In particular, the microscopic cost should have a corresponding mean-field counterpart. This is indeed the case whenever $J_1^N(\x)$ may be written as function acting on its corresponding empirical measure $\mu_{\x}^N$.
    \end{enumerate}
\end{rem}

A direct consequence of assumption \ref{A3} is the continuity of $J_1$ in the Wasserstein metric.
\begin{lem}\label{lem:ass3}
	Assume \ref{A3} and let $\mu,\nu\in \calC([0,T],\calP_2(\R^d))$ with
	\[
	 M_1:=\max_{\ell=1,\ldots,L}\sup_{t\in[0,T]} \left\{|\langle g_\ell,\mu_t\rangle| + |\langle g_\ell,\nu_t\rangle|\right\}<\infty,\qquad M_2:=\sup_{t\in[0,T]}\left\{ \mathfrak{m}_2(\mu_t) + \mathfrak{m}_2(\nu_t)\right\} <\infty.
	\] 
	Then, there exists a constant $C_j>0$, independent of $t\in[0,T]$ such that
	\[
	 |J_1(\mu_t) - J_1(\nu_t)| \le C_j  W_2(\mu_t,\nu_t)\qquad\text{for all\, $t\in[0,T]$}.
	\]
\end{lem}
\begin{proof}
    Let $\mu$ and $\nu\in\calP_2(\R^d)$ be arbitrary. Then, for each $\ell=1,\ldots,L$, we have by \ref{A3}, the mean-value theorem and H\"older's inequality that
    \begin{align*}
        |\langle g_\ell,\mu\rangle - \langle g_\ell,\nu\rangle| &\le \iint_{\R^d\times\R^d} |g_\ell(x)-g_\ell(y)|\dd\pi \\
        &\le \iint_{\R^d\times\R^d} \int_0^1 |\nabla g_\ell|((1-\tau) x + \tau y)|y-x|\dd\tau \dd\pi \\
        &\le \iint_{\R^d\times\R^d} \int_0^1 C_g(1+|(1-\tau) x + \tau y|)|y-x|\dd\tau \dd\pi \\
        &\le C_g\left[ 1 + \left(\sqrt{\mathfrak{m}_2(\mu)} + \sqrt{\mathfrak{m}_2(\nu)}\right)\right] W_2(\mu,\nu),
    \end{align*}
    where $\pi$ is the optimal coupling between $\mu$ and $\nu$. In particular, the estimate above shows that the mapping $\langle g_\ell,\cdot\rangle:\calP_2(\R^d)\to \R$ is locally Lipschitz for every $\ell=1,\ldots,L$.
    
    Denote $p_t = (\langle g_1,\mu_t\rangle,\ldots,\langle g_L,\mu_t\rangle)$ and $q_t = (\langle g_1,\nu_t\rangle,\ldots,\langle g_L,\nu_t\rangle)$. The assumptions on $\mu$ and $\nu$, and the previous estimate yields
    \begin{align*}
        |J_1(\mu_t) - J_1(\nu_t)| &\le \int_0^1 |Dj(q_t + \tau( p_t-q_t))||p_t-q_t|\dd\tau \\
        &\le LC_g \left( 1 + 2\sqrt{M_2}\right)\left(\sup\nolimits_{p\in B_{LM_1}}|Dj(p)|\right)W_2(\mu_t,\nu_t),
    \end{align*}
    where we used the fact that $|(1-\tau)q_t + \tau p_t| \le L M_1$ for all $\tau\in[0,1]$, $t\in[0,T]$.
\end{proof}

\begin{rem}\label{rem:costfun}
	Note that cost functionals that track the center of mass and the variance of a crowd satisfy \ref{A3} and \ref{A4}. In fact, for $\lambda_1,\lambda_2,\lambda_3>0$, 
	\begin{gather*} 
	j(y_1,y_2) = \frac{\lambda_1}{2} |y_1 - x_\text{des}|^2 + \frac{\lambda_2}{4}|y_2 - y_1|^2, \quad g_1(x) = x, \quad g_2(x) = |x|^2, \\ J_2(\con)=\frac{\lambda_3}{2} \sum_{m=1}^M  \int_0^T \left| \frac{\dd}{\dd t} u_t^m \right|^2 \dd t 
	\end{gather*}
	fit into the setting. Therefore, the assumptions are rather general and not restrictive for applications (cf. \cite{sheep1}). 
\end{rem}

The well-posedness of the state problem justifies the notation $\mu(\con)$ assigning the unique solution of the state equation to the control. Then, the optimal control problem we investigate in the following is given by 

\begin{problem}
	Find $\bar\con \in \Uad$ such that
	\begin{equation}\label{OCMF} \tag{${\bf P_\infty}$}
	(\mu(\bar\con),\bar\con) = \argmin_{\mu,\con} J(\mu,\con) \quad \text{subject to } \eqref{longPDE}. 
	\end{equation}
\end{problem}
For later use, we note that in the particle case, i.e.\ for discrete initial data (cf. Remark~\ref{generalization}), we can rewrite the optimization problem as follows:\\
	For $N\in\N$ fixed, find $\bar\con^N \in \Uad$ such that
	\begin{equation}\label{OCN} \tag{${\bf P_N}$}
	(\bar\x^N(\bar\con^N),\bar \con^N) = \argmin_{\x,\con}  J_N(\x,\con) \quad \text{subject to } \eqref{eq:stateODE}. 
	\end{equation}

Using the standard argument based on the boundedness of a minimizing sequence in $\Uad$ and continuity properties of $J$ stated in \ref{A3} and \ref{A4}, we obtain the following existence result:

\begin{thm}
	Assume \ref{A1}--\ref{A4}. Then, the optimal control problem \eqref{OCMF} has a solution $ (\mu(\bar\con),\bar\con) \in \calC([0,T],\calP_2(\R^d)) \times \Uad$.
\end{thm}

\begin{rem}
The well-posedness of \eqref{OCN} follows directly from the above theorem, as the particle problem is a special case of \eqref{OCMF}, see Remark~\ref{generalization}. Nevertheless, one can prove the well-posedness of \eqref{OCN} also directly using classical techniques in the optimal control of ODEs.
\end{rem}

\section{First-order optimality conditions in the Wasserstein space $\calP_2(\R^d)$}\label{sec:FONC}
\noindent
The main objective of this section is to derive the first-order optimality conditions (FOC) for the optimal control problem \eqref{OCMF} in the framework of probability measures with bounded second moment equipped with the 2-Wasserstein distance.
For the sake of a smooth presentation we restrict the interaction terms to the special ones defined in (\ref{eq:structureV}) and (\ref{mfV}), respectively. This allows us to pose the following regularity assumption

\medskip

\begin{enumerate}[label=(A\arabic*)]
	\setcounter{enumi}{4}
	\item \label{A5} $K_1,K_2 \in \calC_b^2(\R^d)$.
\end{enumerate}

\begin{rem}\label{rem:A5}
	Note that assumption \ref{A5} directly implies that $(t,x)\mapsto v(\mu_t,\con_t)(x)$ defined by \eqref{mfV} is an element of $\calC_b^1([0,T]\times \R^d, \R^d)$ for every $\mu\in \calC([0,T],\calP_2(\R^d))$ and $\con\in \calC([0,T],\R^{dM})$ with
	\[
	 K_v:=\sup_{\mu,\con} \bigl\{ \|v(\mu,\con)\|_{\infty} + \|D v(\mu,\con)\|_{\infty}\bigr\} <\infty.
	\]
	In particular, the flow $(t,x)\mapsto Q_t(x)$ is $\calC_b^1([0,T] \times \R^d,\R^d)$ by standard arguments (cf.\ \cite{emmrich}). 
\end{rem}

\medskip

For given initial condition $\hat \mu$ we define the state space as
$$ \mathcal{Y} = \Bigl\{ \mu \in \mathcal{C}([0,T],\calP_2(\R^d))\; \colon\; \mu_t\left|_{t=0}\right. = \hat\mu \in \calP_2^{ac}(\R^d)  \Bigr\}. $$
As the optimization in the $W_2$ setting is not well-known, we begin by discussing known results (see~\cite[Chapter 8.1]{Ambrosio}) regarding the constraint
\begin{equation}\label{eq:PDE}
\partial_t \mu_t + \nabla \cdot (v(\mu_t ,\con_t)\,\mu_t ) = 0, \quad \mu_t\left|_{t=0} \right. = \hat\mu\in \calP_2^{ac}(\R^d).
\end{equation}
Recall Proposition~\ref{prop:exist_cont} that provides for each $\con\in \Uad$ a unique solution $\mu\in \calC([0,T],\calP_2^{ac}(\R^d))$ of \eqref{eq:PDE}. In particular, $\mu$ satisfies
\begin{equation}\label{eq:weak}
E(\mu,\con) [\varphi] := \langle \varphi_T, \mu_T \rangle - \langle \varphi_0, \hat\mu \rangle - \int_0^T \langle \partial_t \varphi + v(\mu_t,\con_t) \cdot \nabla \varphi_t , \mu_t \rangle \dd t = 0, 
\end{equation}
for all $\varphi\in \calA:=\calC_c^1([0,T]\times\R^d)$. Therefore, there is a well-defined solution operator $S\colon \con\mapsto \mu$, which allows us to recast the constrained minimization problem as
\[
 \min \hat J(\con) := J(S\con,\con),\qquad \con\in\Uad,
\]
where $\hat J$ is the so-called reduced functional.

\begin{defi}
	A pair $(\mu,\con)\in \mathcal{Y}\times \Uad$ is said to be {\em admissible} if $E(\mu,\con)[\varphi]=0$ for all $\varphi\in\calA$.
\end{defi}

Unfortunately, the reduced cost functional is not handy in deriving the first-order optimality conditions for~\eqref{OCMF}. For this reason, we will take an \textit{extended-Lagrangian} approach. We begin by observing that \eqref{OCMF} may be recast as
\[
 \min_{(\mu,\con)}\; \calI(\mu,\con), \quad \text{with } \calI(\mu,\con) := \begin{cases}
  J(\mu,\con) & \text{if\; $E(\mu,\con)[\varphi]=0$\; for every\; $\varphi\in \calA$},\\ 
  +\infty & \text{otherwise}
 \end{cases},
\]
which may be further reformulated as
\begin{align}\label{eq:opt_2}
 \min_{(\mu,\con)}\; \calI(\mu,\con) =  \min_{(\mu,\con)} \Bigl\{ J(\mu, \con) + \sup_{\varphi \in \calA} E(\mu,\con)[\varphi] \Bigr\}.
\end{align}
Indeed, notice that $\sup_{\varphi\in \calA} E(\mu,\con)[\varphi]\ge 0$, since $\varphi\equiv 0$ implies $E(\mu,\con)[0]=0$ for every $(\mu,\con)$. Therefore, if $E(\mu,\con)[\varphi]>0$ for some $\varphi$, the linearity in $\varphi$ of $E$ yields $E(\mu,\con)[\alpha\varphi]=\alpha E(\mu,\con)[\varphi]$ for every $\alpha>0$, which consequently shows that $\sup_\varphi E(\mu,\con)[\varphi]=+\infty$.

Under the separation assumption on $J$, i.e.\ $J(\mu,\con)= J_1(\mu)+ J_2(\con)$, \eqref{eq:opt_2} becomes
\[
 \min_{(\mu,\con)}\; \calI(\mu, \con) =  \min_\con \, \biggl\{ J_2(\con) + \min_\mu\sup_{\varphi\in\calA}\, \Bigl\{  J_1(\mu) + E(\mu,\con)[\varphi] \Bigr\}\biggr\} = \min_\con\, \Bigl\{ J_2(\con) + \chi(\con)\Bigr\},
\]
with
\[
 \chi(\con) = \min_\mu\sup_{\varphi\in\calA}\, \Bigl\{  J_1(\mu) + E(\mu,\con)[\varphi] \Bigr\}.
\]
In the following we derive a necessary condition for $(\mu,\con)$ to be a stationary point. Let $( \bar\mu, \bar\con)$ be an optimal pair, $\delta \ge 0$ and $\con^\delta = \bar\con + \delta \h$ be a perturbation of $\bar\con$ for an arbitrary smooth map $\h \colon (0,T) \rightarrow \R^{dM},$ such that $\con^\delta \in \Uad$ and there exists a unique $\mu^\delta \in \mathcal{C}([0,T], \calP_2^{ac}(\R^d))$ satisfying $E(\mu^\delta, \con^\delta)[\varphi] = 0$ for all $\varphi \in \calA.$ Then 
\begin{align*}
 \chi(\con^\delta) &= \min_\mu\sup_{\varphi\in\calA}\, \Bigl\{  J_1(\mu) + E(\mu,\con^\delta)[\varphi]\Bigr\} = J_1(\mu^\delta) \\
 &= J_1(\mu^\delta) - J_1(\bar\mu) + \min_\mu\sup_{\varphi\in\calA}\, \Bigl\{  J_1(\mu) + E(\mu,\bar\con)[\varphi]\Bigr\} \\
 &= J_1(\mu^\delta) - J_1(\bar\mu) + \chi(\bar\con),
\end{align*}
and the directional derivative of $\calG := J_2 + \chi$ at $\bar\con$ along $\h$ is given by
\begin{align*}
\lim_{\delta\to 0}\frac{\calG(\con^\delta)-\calG(\bar\con)}{\delta} = \lim_{\delta\to 0}\frac{[J_1(\mu^\delta) - J_1(\bar\mu)] + [J_2(\con^\delta) - J_2(\bar\con)]}{\delta},
\end{align*}
which requires us to know the relationship between $\mu^\delta$ and $\bar\mu$. 

\begin{rem}
 Note, that Lemma~\ref{lem:forwardDobrushin} above provides a stability estimate of the form
 \[
  W_2(\mu^\delta_t,\mu_t) \le \delta \sqrt{b}e^{aT/2}\|\h\|_{L^2((0,T),\R^{dM})}\qquad\text{for all\, $t\in[0,T]$},
 \]
 for appropriate constants $a,b>0$. Hence, for each $t\in[0,T]$, the curve $[0,\infty)\ni \delta \mapsto \mu^\delta_t \in \calP_2(\R^d)$ starting from $\mu_t$ at $\delta=0$ is absolutely continuous w.r.t.~the 2-Wasserstein distance. In this case, there exists a vector field $\psi_t\in L^2(\mu_t,\R^d)$ for each $t\in[0,T]$ satisfying \cite[Proposition~8.4.6]{Ambrosio}
 \begin{equation}\label{eq:vector-field}
  \lim_{\delta\to 0} \frac{W_2(\mu_t^\delta,(id+\delta \psi_t)_\#\mu_t)}{\delta} = 0.
 \end{equation}
Furthermore,
\[
 W_2^2((id + \delta\psi_t)_\#\mu_t,\mu_t) \le \int_{\R^d} |x + \delta\psi_t(x)-x|^2\dd\mu_t(x) = \delta^2\int_{\R^d} |\psi_t(x)|^2\dd\mu_t(x),
\]
where the explicit coupling $\pi_t = (id+\delta\psi_t,id)_\#\mu_t$ was used.
In particular, we have that
\[ 
 \limsup_{\delta\to 0} \frac{W_2(\mu_t^\delta,\mu_t)}{\delta} = \limsup_{\delta\to 0} \frac{W_2((id+\delta \psi_t)_\#\mu_t,\mu_t)}{\delta}\le \sqrt{\iint_{\R^d\times\R^d} |\psi_t|^2\dd\mu_t}.
\]
\end{rem}
The previous remark allows us to establish an explicit relationship between $\psi_t$ and $\h$.

\begin{lem}\label{lem:4.4}
 Let $(\mu,\con)$ be an admissible pair, $\h\in\calC_c^\infty((0,T),\R^{dM})$ and $\con^\delta=\con + \delta {\bf h}$ such that 
 \begin{enumerate}
  \item[(i)] $\con^\delta\in \Uad$, and
  \item[(ii)] there exists $\mu^\delta\in\calC([0,T],\calP_2^{ac}(\R^d))$ satisfying $E(\mu^\delta,\con^\delta)=0$,
 \end{enumerate}
 for $0<\delta\ll 1$ sufficiently small. If $\psi\in \calC_b^1((0,T)\times \R^d)$ with $\psi_0\equiv 0$ satisfies
  \begin{align}\label{eq:linearization}
  \partial_t\psi_t +  D\psi_t\,v(\mu_t,\con_t) = \calK(\mu_t,\con_t)[\psi_t,\h_t]\qquad \text{for $\mu_t$-almost every $x\in\R^d$},
 \end{align}
 for a bounded Borel map $(t,x)\mapsto \calK(\mu_t,\con_t)[\psi_t,\h_t](x)$ satisfying
 \begin{align}\label{eq:K_limit}
   \lim_{\delta\to 0} \int_0^T \int \left|\frac{v(\nu_t^\delta,\con_t^\delta)\circ(id + \delta\psi_t)(x) - v(\mu_t,\con_t)(x)}{\delta} -  \calK(\mu_t,\con_t)[\psi_t,\h_t](x)\right|^2 \dd\mu_t(x)\dd t = 0,
 \end{align}
 then \eqref{eq:vector-field} holds with this $\psi$, i.e.\
 \[
    \lim_{\delta\to 0} \frac{W_2(\mu_t^\delta,(id+\delta \psi_t)_\#\mu_t)}{\delta} = 0.
 \]
\end{lem} 
\begin{proof}
    For each $t\in[0,T]$, we set $\nu_t^\delta := (id + \delta\psi_t)_\#\mu_t$. We begin by showing that the curve $t\mapsto \nu_t^\delta\in\calP_2(\R^d)$ is absolutely continuous. Due to the assumed regularity on $\psi$ satisfying \eqref{eq:linearization}, the chain-rule applies, and we obtain for any $F\in \calC_c^\infty(\R^d)$, and almost every $t\in(0,T)$,
    \begin{align*}
 \frac{d}{dt}\int F\dd \nu_t^\delta &=  \frac{d}{dt}\int F\circ(id + \delta\psi_t) \dd \mu_t \\
 &= \int \langle(\nabla F)\circ (id + \delta\psi_t), \delta\partial_t\psi_t\rangle \dd \mu_t + \int \langle\nabla (F\circ (id + \delta\psi_t)),v(\mu_t,\con_t)\rangle\dd \mu_t \\
 &= \int \langle (\nabla F)\circ (id + \delta\psi_t), \delta\calK(\mu_t,\con_t)[\psi_t,\h_t] + v(\mu_t,\con_t)\rangle\dd \mu_t \\
 &= \int \langle \nabla F, [\delta \calK(\mu_t,\con_t)[\psi_t,\h_t] + v(\mu_t,\con_t)]\circ (id + \delta\psi_t)^{-1}\rangle \dd\nu_t^\delta =: \int \langle \nabla F, b_t^\delta\rangle \dd\nu_t^\delta\,.
 \end{align*}
Furthermore, by the assumption on $\calK$, we have that
\[
    \int |b_t^\delta|^2\dd\nu_t^\delta = \int |\delta \calK(\mu_t,\con_t)[\psi_t,\h_t] + v(\mu_t,\con_t)|^2 \dd\mu_t <\infty\qquad\text{for almost every $t\in(0,T)$}\,.
\]
Along with the previous computation, we find that $t\mapsto \nu_t^\delta \in \calP_2(\R^d)$ is an absolutely continuous curve satisfying the continuity equation
\[
    \partial_t \nu_t^\delta + \nabla\cdot(b_t^\delta \nu_t^\delta) = 0\qquad\text{in the sense of distributions.}
\]
Consequently, we can consider the temporal derivative of $t\mapsto W_2^2(\mu_t^\delta,\nu_t^\delta)$ to obtain
\begin{align*}
    \frac{1}{2}\frac{d}{dt} W_2^2(\mu_t^\delta,\nu_t^\delta) &= \iint \langle x-y, v(\mu_t^\delta,\con_t^\delta)(x) - b_t^\delta(y)\rangle \dd\pi_t^\delta \\
    &= \iint \langle x-y, v(\mu_t^\delta,\con_t^\delta)(x) - v(\nu_t^\delta,\con_t^\delta)(y)\rangle \dd\pi_t^\delta \\
    &\hspace{8em}+  \iint \langle x-y, v(\nu_t^\delta,\con_t^\delta)(y) - b_t^\delta(y)\rangle \dd\pi_t^\delta =: \mathrm{(I)} + \mathrm{(II)}.
\end{align*}
To estimate $\mathrm{(I)}$, we use assumptions (A1) and (A2) to obtain
\[
    \mathrm{(I)} \le (C_v + C_l) W_2^2(\mu_t^\delta,\nu_t^\delta).
\]
As for $\mathrm{(II)}$, we have 
\begin{align*}
    \mathrm{(II)} &\le W_2(\mu_t^\delta,\nu_t^\delta)\left( \int |v(\nu_t^\delta,\con_t^\delta)(y) - b_t^\delta(y)|^2 \dd\nu_t^\delta\right)^{1/2} \\
    &= W_2(\mu_t^\delta,\nu_t^\delta)\left( \int |v(\nu_t^\delta,\con_t^\delta)\circ(id + \delta\psi_t)(y) - v(\mu_t,\con_t)(y) - \delta \calK(\mu_t,\con_t)[\psi_t,\h_t](y)|^2 \dd\mu_t\right)^{1/2},
\end{align*}
which, together with the estimate for $\mathrm{(I)}$, gives
\begin{align*}
    \frac{d}{dt} W_2^2(\mu_t^\delta,\nu_t^\delta) \le CW_2^2(\mu_t^\delta,\nu_t^\delta) + \delta^2 \mathsf{e}_t^\delta,
\end{align*}
for some constant $C>0$, and where
\[
    \mathsf{e}_t^\delta :=  \int \left|\frac{v(\nu_t^\delta,\con_t^\delta)\circ(id + \delta\psi_t)(y) - v(\mu_t,\con_t)(y)}{\delta} -  \calK(\mu_t,\con_t)[\psi_t,\h_t](y)\right|^2 \dd\mu_t.
\]
Since $W_2(\mu_0^\delta,\nu_0^\delta)=0$, an application of Gronwall's inequality yields
\[
    \sup_{t\in[0,T]} \frac{W_2^2(\mu_t^\delta,\nu_t^\delta)}{\delta^2} \le e^{CT}\int_0^T \mathsf{e}_s^\delta\dd s \longrightarrow 0\qquad\text{as $\delta \to 0$},
\]
due to the assumption on $\calK$ in \eqref{eq:K_limit}, thereby concluding the proof.
\end{proof}

\begin{rem}\label{rem:K}
	We mention that for any ${\bf h}\in\calC_c^\infty((0,T),\R^{dM})$ and any sufficiently smooth mapping $F\colon \R^d\to\R^d$, we have for ${\bf x}^\delta = {\bf x} + \delta{\bf h}$:
	\[
	 F({x}^{\delta,m}) = F({ x}^{m}) + \delta (D F)({ x}^{m})[h^m] + O(\delta^2)\qquad \text{for\; $m=1,\dots,M$}.
	\]
 In particular, for the velocity field $v$ given in \eqref{mfV} one deduces
  \begin{align}\label{eq:rem:K}
  \calK(\mu,\con)[\psi,\h] &= Dv(\mu,\con)\psi + \int (D K_1)(\cdot -y)\psi(y)\dd  \mu(y) + \sum_{m=1}^M (DK_2)(\cdot-u^{m})\,h^m,
  \end{align}
  which satisfies
  \[
    \sup_{t\in(0,T)}\|\calK(\mu_t,\con_t)[\psi_t,\h_t]\|_\infty \le C\Bigl(\|\psi\|_{L^\infty((0,T)\times\R^d)} + \|\h\|_{L^\infty((0,T))}\Bigr).
  \]
  From assumption (A5), it is not difficult to see that \eqref{eq:K_limit} is satisfied.
\end{rem}

The existence of a $\psi\in \calC_b^{1}((0,T)\times \R^d)$ satisfying the assumptions of Lemma~\ref{lem:4.4} is provided in the following statement.
\begin{thm}
Let the assumptions of Lemma~\ref{lem:4.4} hold. For the velocity field $v: \calP_2(\R^d) \times \R^{dM} \to \Lip_{loc}(\R^d)$ given in \eqref{mfV} there exists $\psi \in \calC_b^{1}((0,T)\times \R^d)$ with $\psi_0 = 0$ satisfying
\begin{equation*}
  \partial_t\psi_t +  D\psi_t\,v(\mu_t,\con_t) = \calK(\mu_t,\con_t)[\psi_t,\h_t]\qquad \text{for $\mu_t\,dt$-almost every $(t,x)\in(0,T)\times\R^d$},
 \end{equation*}
 where $\calK$ is given in \eqref{eq:rem:K}.
\end{thm}
\begin{proof}
We consider $\Gamma = \calC([0,T],\calC_b^1(\R^d,\R^d))$  and the operator
\[
\Gamma\ni\omega \mapsto H(\omega) \quad\text{with}\quad H(\omega)(t,x) = \int_0^t \calK(\mu_s, \con_s) [\omega_s,\h_s] (Q_s(x,t))\, ds. 
\]
 First, we have to show that $H(\omega) \in \Gamma.$ Due to (15), (A5) and the properties of the flow discussed in Remark~\ref{rem:A5}, we have $DH(\omega)(t) \in \calC_b(\R^d)$ and continuous w.r.t.~$t$. Therefore, it holds $H(\omega) \in \calC([0,T],\calC_b^1(\R^d)).$ In particular, $H \colon \Gamma \rightarrow \Gamma$ is well-defined. 

To establish the contraction property of $H,$ we equip $\calC([0,T],\calC_b^1(\R^d,\R^d))$ with the weighted norm
$$ \| \omega \|_{\text{exp}} := \max\limits_{t \in [0,T]} \Big\{ e^{-\lambda} \big( \| \omega(t) \|_\text{sup} + \| D\omega(t) \|_\text{sup} \big) \Big\}
$$
for some $\lambda > 0$ to be specified below.
Note that $(\calC([0,T],\calC_b^1(\R^d,\R^d)), \| \cdot \|_{\text{exp}})$ is complete. 

Using the structure of $\calK$ in Remark~\ref{rem:K}, we obtain
\begin{align*}
    |H(\omega^1) - H(\omega^2)|(t,x) &\le \int_0^t |\calK(\mu_s,\con_s)[\omega_s^1 - \omega_s^2, \h_s] (Q_s(x,t))|\, ds \\
    &\le \int_0^t (\| Dv \|_{\sup} + \| DK_1 \|_{\sup}) \| \omega_s^1 - \omega_s^2\| ds.
\end{align*}
As for the space derivative we obtain
\begin{align*}
    |D H(\omega^1) - &D H(\omega^2)|(t,x) \le \int_0^t |D \calK(\mu_s,\con_s)[\omega^1_s-\omega^2_s ,\h_s](Q_s(x,t))| |DQ_s(x,t)| ds \\
    &\le \int_0^t \Bigl(\| D^2v \|_{\sup} + \| Dv\|_{\sup} + \| D^2K_1 \|_{\sup}\Bigr) \Bigl(\| \omega_s^1 - \omega_s^2 \|_\text{sup} + \| D\omega_s^1 - D\omega_s^2 \|_\text{sup}\Bigr)\, ds.
\end{align*}
We define $C_v = 2\| Dv \|_\infty + \| DK_1 \|_\infty + \| D^2v \|_\infty  + \| D^2K_1 \|_\infty$ and add the two inequalities to obtain
\begin{align*}
    &|H(\omega^1) - H(\omega^2)|(t,x) + |D H(\omega^1) - D H(\omega^2)|(t,x) \\ 
    &\qquad\le \int_0^t C_v \Bigr(\| \omega_s^1 - \omega_s^2 \|_\text{sup} + \| D\omega_s^1 - D\omega_s^2 \|_\text{sup}\Bigl)\, ds
    \le \frac{C_v}{\lambda} e^{\lambda t} \| \omega_1 - \omega_2\|_\text{exp} \,.
\end{align*}
Multiplying each of the above estimates with $e^{-\lambda}$ and taking the  supremum over $t$ and $x$ leads to
\begin{equation*}
\| H(\omega^1) - H(\omega^2) \|_\text{exp} \le \frac{C_v}{\lambda} \|\omega^1 - \omega^2 \|_\text{exp}.  
\end{equation*}
Choosing $\lambda > C_v$ allows us to conclude the contraction property of $H.$ An application of the Banach fixed-point theorem yields a solution $\psi \in \calC([0,T],\calC_b^1(\R^d,\R^d))$ given by
\[
    \psi_t(x) = \int_0^t \calK(\mu_s,\con_s)[\psi_s,\h_s](Q_s(x,t)) \dd s.
\]
It is straight forward to see that
\[
    \Gamma\cap \calC^1((0,T), \calC_b(\R^d,\R^d)) \hookrightarrow \calC_b^1((0,T)\times\R^d,\R^d).
\]
Finally, a direct computation shows that $\psi$ satisfies the evolution equation.
\end{proof}

Now, we are able to state the first-order necessary condition for $(\mu,\con)$ to be a stationary point.
\begin{thm}\label{thm:optW}
 Let $(\bar\mu,\bar\con)$ be an optimal pair, $J_2$ be G\^ateaux-differentiable, and $J_1$ a cylindrical function of the form given in \ref{A3}. Then, for any $\h \in\calC_c^\infty((0,T),\R^{dM})$ it holds:
 \begin{align}\label{eq:optimality}
dJ_2(\bar\con)[\h] + \int_0^T \!\!\int \langle \delta_\mu J_1(\bar\mu_t), \psi_t\rangle\dd \bar\mu_t\dd t = 0,
 \end{align}
 where
 \begin{align}\label{eq:deltaJ_1}
   \delta_\mu J_1(\mu)(x) := \sum_{\ell=1}^L(\partial_\ell j)(\langle g_1,\mu \rangle,\ldots,\langle g_L,\mu\rangle)(\nabla g_\ell)(x),
 \end{align}
 and $t\mapsto\psi_t\in L^2(\mu_t,\R^d)$ satisfying \eqref{eq:linearization} with initial condition $\psi_0=0$. 
\end{thm}

\begin{proof}
Since $J_1(\mu)$ is a cylindrical function, we have that
 \begin{align*}
 J_1(\mu_t^\delta)-J_1(\bar\mu_t) &= J_1((id+\delta\psi_t)_\#\bar\mu_t)-J_1(\bar\mu_t) + o(\delta)\\
 &= \delta\int \sum_{\ell=1}^L(\partial_\ell j)(\langle g_1,\bar\mu_t\rangle,\ldots,\langle g_L,\bar\mu_t\rangle)\langle\nabla g_\ell, \psi_t\rangle\dd \bar\mu_t + o(\delta),
 \end{align*}
 where $\psi$ satisfies \eqref{eq:linearization} with $\psi_0=0$. Therefore, owing to the minimality of $\bar\con$, we find
 \begin{align*}
 0\le \frac{\calG(\con^\delta)-\calG(\bar\con)}{\delta} &= dJ_2(\bar\con)[{\h}] 
 + \int_0^T \!\!\int \langle \delta_\mu J_1(\bar\mu_t), \psi_t\rangle\dd \bar\mu_t \dd t + O(\delta).
 \end{align*}
 Passing to the limit $\delta\to 0+$ yields
 \[
 0\le dJ_2(\bar\con)[{\h}] + \int_0^T \!\!\int \langle \delta_\mu J_1(\bar\mu_t), \psi_t\rangle\dd \bar\mu_t \dd t,
 \]
 for any ${\h}\in\calC_c^\infty((0,T),\R^{dM})$. Notice, however, that changing the sign of ${\h}$ leads to a change of sign of $\psi$, which then provides the equality \eqref{eq:optimality}.
\end{proof}

In order to provide an adjoint-based first-order optimality system, we now derive the equation for the dual variable. We consider the dual problem corresponding to \eqref{eq:linearization}, by testing the equation \eqref{eq:linearization} with a family of vector-valued measures $(m_t)_{t\in(0,T)}$ to obtain
\begin{align*}
 \int_0^T \!\!\int \Bigl(\partial_t\psi_t + D\psi_t\,v(\bar\mu_t,\bar\con_t)  - \calK(\bar\mu_t,\bar\con_t)[\psi_t,h_t]\Bigr)\cdot \dd m_t \dd t = 0,
\end{align*}
and set $\calK(\mu,\con)[\psi,h] = \calK^1(\mu,\con)[\psi] + \calK^2(\con)[\h]$, where
\begin{align*}
\calK^1(\mu,\con)[\psi] &:= D v(\mu,\con)\psi +  \int (D K_1)(\cdot -y)\psi(y)\dd \mu(y),\\
\calK^2(\con)[\h] &:= \sum\nolimits_{\ell} (DK_2)(\cdot-u^{\ell})\,h^{\ell}.
\end{align*}
Using $\psi_0=0$ and integrating by parts, we obtain
\begin{align*}
  &\int_0^T \int \langle \partial_tm_t + \nabla\cdot\bigl(v(\bar\mu_t,\bar\con_t)\otimes m_t\bigr) + \calK^{1,*}(\bar\mu_t,\bar\con_t)[m_t],\psi_t\rangle\dd t \\
 &\hspace*{16em}=\int \psi_T\cdot d m_T - \int_0^T \int \calK^2(\bar\con_t)[{\h_t}]\cdot \dd m_t\dd t,
\end{align*}
where
\begin{align}\label{eq:K1_dual}
\calK^{1,*}(\mu,\con)[m] = \nabla v(\mu,\con)\, m + \mu\int  (\nabla K_1)(y-\cdot)\dd m(y).
\end{align}  
By choosing $\bar m$ to satisfy the dual problem
\begin{align}\label{eq:adjoint}
 \partial_t \bar m_t + \nabla\cdot\bigl(v(\bar\mu_t,\bar\con_t)\otimes \bar m_t\bigr) + \calK^{1,*}(\bar\mu_t,\bar\con_t)[\bar m_t] = \bar\mu_t\delta_\mu J_1(\bar\mu_t),
\end{align}
subject to the terminal condition $\bar m_T = 0,$ 
we find with the help of the optimality condition~\eqref{eq:optimality}, that
\begin{equation}\label{eq:optconditionPDE}
 dJ_2( \bar\con)[\h] - \int_0^T \int \calK^2(\bar\con_t)[\h_t] \cdot \dd\bar m_t\dd t = 0\qquad \text{for all ${\h}\in\calC_c^\infty((0,T),\R^{dM})$}.
\end{equation}

\begin{rem}\label{rem:W}
	If $|\bar m_t|\ll \bar\mu_t$ for every $t\in[0,T]$, i.e.\ there is a vector field $\bar\xi_t:\R^d\to\R^d$ such that $\bar m_t = \bar \xi_t\bar\mu_t$, where $\bar\mu$ satisfies \eqref{eq:PDE}, then equation \eqref{eq:adjoint} formally reduces to
	\begin{align}\label{eq:xiW}
	 \partial_t \bar\xi_t + D\bar\xi_t\,v(\bar\mu_t,\bar\con_t) = - \nabla v(\bar\mu_t,\bar\con_t)\,\bar\xi_t - \int (\nabla K_1)(y-\cdot)\,\bar\xi_t(y)\dd  \bar\mu_t(y) + \delta_\mu J_1(\bar\mu_t).
	\end{align}
\end{rem}

\begin{rem}\label{rem:adjoint_potential}
	If we further assume that $K_1$ and $K_2$ are gradients of potential fields, then $\nabla K_1$ and $\nabla K_2$ are symmetric and the previous equation takes the simpler form
	\[
		\partial_t \bar\xi_t + \nabla( v(\bar\mu_t,\bar\con_t)\cdot \bar\xi_t) = -\int \bar\xi_t(y)\cdot(\nabla K_1)(y-\cdot)\dd  \bar\mu_t(y) + \delta_\mu J_1(\bar\mu_t).
	\]
	In this case, one can expect $\bar\xi$ to be a gradient of a potential field (compare also the results in \cite{herty18}), i.e.\ $\bar\xi=\nabla\bar\phi$ for a function $\bar\phi$ satisfying the scalar equation
	\begin{equation}\label{eq:compL2adjoint}
		\partial_t \bar\phi_t + v(\bar\mu_t,\bar\con_t)\cdot \nabla \bar\phi_t = \int \nabla\bar\phi_t(y)\cdot K_1(y-\cdot)\dd  \mu_t(y) + \sum_{i=1}^L(\partial_i j)(\langle g_1,\mu_t \rangle,\ldots,\langle g_L,\mu_t\rangle) g_i.
	\end{equation}
\end{rem}

\subsection{Well-posedness of the adjoint equation}
To obtain the well-posedness of the adjoint equation \eqref{eq:adjoint} we make use of equation \eqref{eq:xiW}. Indeed, due to assumptions \ref{A3} and \ref{A5}, we can make use of the method of characteristics and Banach's fixed-point theorem.

\begin{thm}
	 Let assumptions \ref{A1}-\ref{A5} hold and $(\mu,\con)$ be admissible with initial condition $\hat{\mu}\in\calP(\R^d)$ having compact support. Then, the equation
	 \[
	  \partial_t \xi_t + D\xi_t\,v(\mu_t,\con_t) = \Psi(\mu_t,\con_t)[\xi_t],\qquad \xi_T= p\in \calC_b(\R^d,\R^d),
	 \]
	 with
	 \begin{align}\label{eq:Psi}
	  \Psi(\mu,\con)[\xi] = - \nabla v(\mu,\con)\,\xi - \int (\nabla K_1)(y-\cdot)\,\xi(y)\dd \mu(y) + \delta_\mu J_1(\mu)
	 \end{align}
	 has a unique solution $\xi\in \calC([0,T]\times \R^d,\R^d)$ with the representation
	 \begin{equation}\label{eq:CharPhi}
	 	\xi_t(x) = p(Q_T(x,t)) - \int_t^T \Psi(\mu_s,\con_s)[\xi_s](Q_s(x,t))\dd s,
	 \end{equation}
	 where $Q$ satisfies \eqref{eq:W-Flow}. In particular, $m = \xi\mu$ yields a distributional solution of \eqref{eq:adjoint}.
\end{thm}
\begin{proof}
	We begin by recalling that the Lagrangian flow satisfies 
	\begin{gather*}
		Q_\cdot(\cdot,t)\in \calC(\R^d\times [t,T],\R^d)\qquad\text{for every\, $t\in[0,T)$},\\
		\exists\, \Omega\subset\R^d\;\text{compact}:\quad Q_s(x,t)\in \Omega\qquad\text{for all\, $t\in[0,T)$, $s\in[t,T]$ and $x\in\text{supp}(\hat\mu)$}.
	\end{gather*}
 For any $\omega \in \Gamma:=\calC([0,T]\times\R^d,\R^d)$, we define the operator
 \[
  H(\omega)(t,x):= p(Q_T(x,t)) - \int_t^T \Psi(\mu_s,\con_s)[\omega_s](Q_s(x,t))\dd s.
 \]
 Observe that $H(\omega)\in \Gamma$ due to the properties of the Lagrangian flow, and the fact that $p\in \calC_b(\Omega)$ and $K_1, K_2\in \calC_b^1(\R^d)$ by assumption \ref{A5}. In particular, $H\colon \Gamma\to \Gamma$ is a well-defined mapping.
 
 To show that $H$ is  a contraction on $\Gamma$, we first define a norm on $\Gamma$ given by
 \[
 	\norm{\omega}_{\exp} := \sup \big\{ e^{-4c_K (T-t)} \|\omega_t\|_{\sup} \; \colon \; t \in (0,T) \big\},
 \]
 where $c_K = \norm{DK_1}_{\sup} +  \norm{DK_2}_{\sup}.$ We note that $(\Gamma, \norm{\cdot}_{\exp})$ is complete and the estimate
 \begin{align*}
	|H(\omega^1) - H(\omega^2)|(t,x) &\le \int_t^T |\nabla v(\mu_s,\con_s)|(Q_s(x,t))|\omega^1_s - \omega^2_s|(Q_s(x,t))\dd s \\
	&\hspace*{4em} + \int_t^T \int_{\R^d} (\nabla K_1)(y-Q_s(x,t))\,|\omega_s^1 - \omega_s^2|(y)\dd \mu_s(y)\dd s \\
	&\le 2c_K\|\omega^1 - \omega^2\|_{\exp}\int_t^Te^{4c_K(T-s)}\dd s \\
	&= (1/2)\|\omega^1 - \omega^2\|_{\exp}(e^{4c_K(T-t)}-1)
 \end{align*}
 holds true for any $\omega^1$, $\omega^2\in\Gamma$. Taking the supremum over $x\in\R^d$ in the inequality above, multiplying with $e^{-4c_K(T-t)}$ and then taking the supremum over $t\in[0,T]$ yields
 \[
  \|H(\omega^1) - H(\omega^2)\|_{\exp} \le (1/2)\|\omega^1 - \omega^2\|_{\exp}.
 \]
 Therefore, the Banach fixed-point theorem provides a unique $\xi\in \Gamma$ satisfying \eqref{eq:CharPhi}.
\end{proof}

Summarizing the above computations, we end up at the following result.
\begin{thm}\label{eq:opt_W}
 A minimizing pair $(\bar\mu,\bar\con)$ of the problem \eqref{OCMF} satisfies
 \begin{align*}
 \partial_t \bar\mu_t + \nabla\cdot(\bar\mu_t\, v(\bar\mu_t, \bar\con_t)) &= 0,  \\
   \delta_{\bar\con} J_2(\bar\con) &= \frac{1}{\lambda} \int_{\R^d} (\nabla K_2)(x-\bar u_t^{\ell})\dd  \bar m_t(x),
 \end{align*}
 where the adjoint variable $\bar m$ satisfies
\begin{align*}
 \partial_t\bar m_t + \nabla\cdot(v(\bar\mu_t,\bar\con_t)\otimes \bar m_t) &=  -\nabla v(\bar\mu_t,\bar \con_t)\bar m_t  - \bar\mu_t \int_{\R^d} (\nabla K_1)(y-x)\dd  \bar m_t(y)\\
&\hspace*{8em}+\bar\mu_t\sum_{i=1}^k (\partial_i j)(\langle g_1 ,  \bar\mu_t \rangle, \dots, \langle g_k,  \bar\mu_t \rangle) \nabla g_i
 \end{align*}
 subject to the conditions 
 \[
  \bar\mu_t|_{t=0}=\hat\mu,\qquad \bar m_t|_{t=T} = 0,\qquad  \bar\con_t|_{t=0}= \hat\con,\qquad \left. \frac{d \bar\con_t}{dt}\right|_{t=T}=0.
\]
\end{thm}

Note, that in the case of the cost functional given in Remark~\ref{rem:costfun} the optimality condition turns out to be a boundary value problem in time. In fact, we obtain as explicit representation
\begin{align*}
d_{u^\ell} J_2(\con)[h^\ell] &= \lambda\int_0^T \left\langle \frac{d}{dt} u_t^\ell, \frac{d}{dt} h_t^\ell \right\rangle_{L^2}\;dt\\
&= \lambda\,\left[  \frac{d}{dt} u_t^\ell \cdot h_t^\ell  \right]_0^T - \lambda\,\int_0^T \left\langle \frac{d^2}{dt^2} u_t^\ell, h_t^\ell\right\rangle_{H^{-1},H^1} \dd t
\end{align*}
for  $h = (h^\ell)_{\ell=1,\dots,M} \in H^1((0,T),\R^{dM})$ with $h_0 =0$. In particular, the variational lemma yields
\begin{gather*}
\delta_{u_t^\ell} J_2(u_t) = \frac{d^2}{dt^2} u_t^{\ell} = \int_{\R^d} (\nabla K_2)(x-u_t^{\ell})\dd  m_t(x) \quad \text{ in } H^{-1}((0,T),\mathbb{R}^d)\\
u_0^{\ell} = \hat \con_0^{\ell} \quad \text{ and } \quad   \frac{d}{dt} u_T^{\ell} =0 \quad \text{for all}\quad \ell=1,\dots,M \text{ and } \con \in \Uad.
\end{gather*}

\section{Relations between first-order optimality systems}\label{sec:relations} \noindent
In order to discuss the links of first-order optimality system in the space of probability measures derived in the previous section to the one of the ODE constrained problem and the optimization problem based on a classical $L^2$-approach, we shall give the respective first-order optimality systems in the interest of completeness. 
\subsection{First-order optimality conditions in the microscopic setting}\label{sec:FONCmicro}
We derive the first-order optimality conditions for the microscopic case by the classical $L^2$-approach. Again, the set of admissible controls $\Uad$ is defined as above. The state space $Y$ is the Hilbert space
\[
	Y = H^1((0,T),\mathbb{R}^{Nd})\hookrightarrow \calC([0,T],\R^{Nd}).
\]
Further, we define
\[
	Z:= L^2((0,T), \mathbb{R}^{Nd}) \times \mathbb{R}^{Nd}
\]
the space of Lagrange multipliers with the dual $Z^*=Z$. This allows us to define the state operator $e_N \colon Y \times \Uad \rightarrow Z$ for the microscopic system as
\[
	e_N(\x,\con) = \begin{pmatrix} \frac{\dd}{\dd t}{\x_t} - v^N(\x_t,\con_t) \\ \x_0 - \hat\x \end{pmatrix},
\]
and the weak form
\[
	\langle e_N(\x,\con), (\boldsymbol\xi,\boldsymbol\eta) \rangle_{Z} = \int_0^T \left(\frac{\dd}{\dd t}{\x_t} - v^N(\x_t,\con_t) \right) \cdot \boldsymbol\xi_t \dd t + (\x_0 - \hat\x) \cdot \boldsymbol\eta. 
\]
We note that due to $Y \hookrightarrow \calC([0,T], \R^{dN})$  the evaluation of $\x_0$ is justified. Let $(\boldsymbol\xi,\boldsymbol\eta)\in Z$ denote the Lagrange multipliers. Then, the Lagrangian corresponding to \eqref{OCN} with $N\in\N$ fixed reads
\[
	\mathcal{L}_\text{micro}^N(\x,\con,\boldsymbol\xi,\boldsymbol\eta) = NJ_N(\x,\con) + \langle e_N(\x,\con), (\boldsymbol\xi,\boldsymbol\eta) \rangle_{Z}.
\]

\begin{rem}
 Note that the $J_N$ is multiplied with $N$ to obtain the appropriate balance between the two terms in the Lagrangian as $N\to\infty$.
\end{rem}

As usual, the first-order necessary optimality condition is derived by solving 
\[
	\dd \mathcal{L}_\text{micro}^N(\x,\con,\boldsymbol\xi,\boldsymbol\eta)\overset{!}{=} 0.
\]
Exploiting \ref{A3}--\ref{A5} we can calculate for any $h = (h_\x,h_\con) \in Y \times \Uad$ the following G\^ateaux derivatives of the cost functional
\begin{gather*}
	\dd_\x J_N(\x,\con)[h^\x] = \int_0^T d_\x J_1^N(\x_t)[h_t^\x] \dd t, \qquad \qquad
	\dd_\con J_N(\x,\con)[h^\con] =  \int_0^T d_\con J_2(\con_t)[h_t^\con] \dd t,
\end{gather*}
and for the second part of the Lagrangian
\begin{subequations}\label{eq:derivativesODE}
	\begin{gather}
	\langle \dd_\x e_N(\x,\con)[h^\x],(\boldsymbol\xi,\boldsymbol\eta) \rangle = \int_0^T \left( \frac{\dd}{\dd t} h_t^\x - D_\x v^N(\x_t,\con_t)[h_t^\x] \right) \cdot \boldsymbol\xi_t \dd t + h_0^\x \cdot \boldsymbol\eta, \\
	\langle \dd_\con e_N(\x,\con)[h^\con],(\boldsymbol\xi,\boldsymbol\eta) \rangle = -\int_0^T D_\con v^N(\x_t,\con_t)[h_t^\con] \cdot \boldsymbol\xi_t \dd t. 
	\end{gather}
\end{subequations}
Assuming further, that $\boldsymbol\xi \in Y$, one may formally derive the strong formulation of the adjoint system. Indeed, using integration by parts we arrive at the following result.
\begin{thm}
	Let $(\bar x^N,\bar u^N)$ be an optimal pair. The optimality condition corresponding to \eqref{OCN}, with $N\in\N$ fixed, reads
	\begin{equation}\label{eq:optconditionODE}
	\int_0^T N d_\con J_2(\bar\con^N_t)[h_t^\con]  - D_\con v^N(\bar\x^N_t,\bar\con^N_t)[h_t^\con] \cdot \bar{\boldsymbol\xi}^N _t \dd t = 0 \quad \text{for all } h^\con \in \calC_c^\infty((0,T),\R^{dM}),
	\end{equation}
	where $\bar{\boldsymbol\xi}^N \in Y$ satisfies the adjoint system given by
	\begin{equation}\label{eq:adjointODE}
	\frac{\dd}{\dd t} \bar{\boldsymbol\xi}^N_t= -\nabla_\x v^N(\bar\x^N_t,\bar\con_t^N)\, \bar{\boldsymbol\xi}^N_t  + N\nabla_\x J_1^N(\bar\x_t^N)
	\end{equation}
	supplemented with the terminal condition $\bar\xi^N_T = 0.$ 
\end{thm}

Similar to the previous case, we obtain for the cost functional given in Remark~\ref{rem:costfun} the boundary value problem
\begin{gather*}
\frac{d^2}{dt^2} u_t^{\ell} = \frac{1}{\lambda N} \sum_{i=1}^N \nabla K_2(x_t^{N,i} - u_t^{\ell})\,\xi_t^{N,i} \quad \text{ in } H^{-1}((0,T),\mathbb{R}^d)\\
u_0^{\ell} = \hat \con_0^{\ell} \quad \text{ and } \quad   \frac{d}{dt} u_T^{\ell} =0 \quad \text{for all}\quad \ell=1,\dots,M \text{ and } \con \in \Uad.
\end{gather*}
Further, for the special structure of the interaction forces defined in (\ref{eq:structureV}) and $J$ given by \ref{A3} we obtain for the adjoint equation
\begin{align}\label{eq:adjointODE_explicit}
\begin{aligned}
\frac{\dd}{\dd t} \xi^i_t &= \frac{1}{N}\,\sum_{j=1}^N\nabla K_1(x_t^i-x_t^j)\xi_t^i - \frac{1}{N}\,\sum_{j=1}^N\nabla K_1(x_t^j-x_t^i)\xi_t^j + \sum_{\ell=1}^M \nabla K_2(x_t^i-u_t^\ell)\,\xi_t^i\\
&\hspace*{6em}+ \sum_{l=1}^L \partial_l j \bigl( \langle g_1,\mu_t^N\rangle, \ldots, \langle g_L,\mu_t^N\rangle \bigr)\nabla g_l(x_t^i), \quad i =1,\ldots,N,
\end{aligned}
\end{align}
with terminal condition $\xi^i_T =0$.

\begin{rem}\label{rem:adjointODE_N}
	Using a similar idea as in the proof in the Appendix (Gronwall inequality), it is not difficult to see that under assumption \ref{A5}, $\boldsymbol{\xi}^N$ satisfying \eqref{eq:adjointODE_explicit} enjoys the uniform bound
	\[
	 \sup_{t\in[0,T]} \frac{1}{N} \sum_{i=1}^N |\xi_t^{N,i}|^2 =: C_\xi <\infty,
	\]
	where $C_\xi>0$ is independent of $N\in \N$, and depends only on $DK_i$, $Dj$ and $D g$.
\end{rem}

\begin{rem}\label{rem:N-m_t}
 Defining the vector-valued measure $m_t^N:= (1/N)\sum_{i=1}^N \xi_t^i\delta_{x_t^i}$, we have by construction that $m_t$ satisfies
 \begin{align*}
	\frac{d}{dt} \int_{\R^d}\nabla\varphi \cdot d m_t^N &= -\int_{\R^d} \nabla\varphi \cdot \nabla v(\mu_t^N,\con_t^N)\dd m_t^N- \int_{\R^d}\nabla\varphi\cdot\int_{\R^d} \nabla K_1(y-\cdot)\dd m_t^N(y)\dd \mu_t^N\\
	&+\sum_{l=1}^L \int_{\R^d} \partial_l j \bigl( \langle g_1,\mu_t^N\rangle, \ldots, \langle g_L,\mu_t^N\rangle \bigr)\nabla\varphi\cdot \nabla g_l\dd \mu_t^N +\int_{\R^d}\nabla^2\varphi\, v(\mu_t^N,\con_t^N)\cdot d m_t^N 
 \end{align*}
for all $\varphi\in\calC_c^\infty(\R^d)$. In other words, $m_t^N$ is a distributional solution of the equation
\begin{align}\label{eq:N-m_t}
\begin{aligned}
	\partial_t m_t + \nabla\cdot\bigl(v(\mu_t^N,\con_t^N)\otimes m_t\bigr) &= -\nabla v(\mu_t^N,\con_t^N)\,m_t - \mu_t^N\int_{\R^d} \nabla K_1(y-\cdot)\dd m_t(y) \\
	&\hspace*{8em}+\mu_t^N\sum_{l=1}^L \partial_l j \bigl( \langle g_1,\mu_t^N\rangle, \ldots, \langle g_L,\mu_t^N\rangle \bigr) \nabla g_l.
\end{aligned}
\end{align}
\end{rem}

We emphasize that \eqref{eq:N-m_t} coincides with the adjoint equation in the mean-field setting \eqref{eq:adjoint}.

\subsection{First-order optimality conditions in the mean-field setting: $L^2$-approach}\label{sec:FONCmf}
To be able to work in the classical $L^2$-setting, we will need additional assumptions to obtain Lebesgue integrable solutions:

\begin{enumerate}[label=(A\arabic*)]
	\setcounter{enumi}{5}
	\item \label{A6} There exists a compact $\Omega_0 \subset \R^d$ such that the initial condition satisfies $\text{supp}(\hat\mu) \in \Omega_0.$ 
	\item \label{A7} The initial measure $\hat\mu$ has a Lebesgue density $\hat f \in L^2(\Omega_0)$.
\end{enumerate}

In particular, \ref{A5}-\ref{A7}  ensure the boundedness of the support of $\mu_t$ for all times $t \in [0,T]$. Hence, we can  fix a bounded domain $\Omega \subset \R^d$ with smooth boundary containing the support of $\mu_t$ for all times $t \in [0,T].$
In this section we strongly use that $\mu$ is absolutely continuous w.r.t.\ the Lebesgue measure and denote its density by $f_t=d\mu_t/dx$ with initial condition $f_0= \hat f$. Then, we define the state space of the PDE optimization problem as
\begin{equation*}
\mathcal Y = \Big\{ f \in L^2((0,T), H^1(\Omega)) \;\colon \partial_t f \in L^2((0,T), H^{-1}(\Omega)) \Big\}.
\end{equation*}
Let $\mathcal X = L^2((0,T), H^{1}(\Omega))$ and $\mathcal Z = \mathcal X \times  L^2(\Omega)$ be the space of adjoint states with dual $\mathcal Z^*.$
The control space $\Uad$ was already defined in \eqref{eq:Uad}. For the derivation of the adjoints we consider here only  the special case given by (\ref{eq:structureV}) and \ref{A4}.
We define the mapping $e_\infty \colon \mathcal{Y}\times \Uad \rightarrow \mathcal{Z}^*$ by
\begin{align*}
\langle e(\mu,\con),(q,\eta) \rangle_{\mathcal{Z}^*,\mathcal{Z}} &= \int_0^T \langle \partial_t f_t,  q_t \rangle_{H^{-1}, H^1} +  \int_\Omega  \nabla \cdot \big( v(f_t,\con_t) f_t \big) q_t \dd x \dd t 
- \int_\Omega (f_0 - \hat f) \eta \dd x
\end{align*}
with adjont state $(q,\eta) \in \mathcal{Z}$. The Lagrangian corresponding to \eqref{OCMF} reads
\begin{equation*}
\mathcal{L}_\text{macro}(\mu,\con,q,\eta) = J(\mu,\con) + \langle e(\mu,\con),(q,\eta) \rangle_{\mathcal{Z}^*,\mathcal{Z}}.
\end{equation*}
Analogously to the microscopic case, we derive the adjoint system and the optimality condition by calculating the derivatives of $\mathcal{L}_\text{meso}$ w.r.t.~the state variable and the control. The standard $L^2$-calculus yields
\begin{equation*}
\dd_f J(\mu,\con)[h^f] = \int_0^T d_f J_1(\mu_t) [h_t^f] \dd t , \qquad 
\dd_\con J(\mu,\con)[h^\con] = \int_0^T d_\con J_2(\con_t)[h_t^\con] \dd t
\end{equation*}
for the cost functional and
\begin{align*}
\langle \dd_f e(\mu,\con)[h^f] , (q,\eta) \rangle &=  \int_0^T \langle  \partial_t h_t^f , q_t \rangle_{H^{-1}, H^1} + \left\langle \int_\Omega K_1(y-x) \cdot \nabla q_t(y) f_t(y)\dd y, h_t \right\rangle \dd t \\
& \qquad -\int_0^T \langle v(f_t,\con_t) \cdot \nabla q , h_t \rangle dt - \langle h_0, \eta \rangle,\\
\langle \dd_\con e(\mu,\con)[h^\con] , (q,\eta) \rangle &= -\int_0^T \int_\Omega  D_\con v(f_t,\con_t)[h_t^\con]\cdot \nabla q_t\,f_t \dd x \dd t 
\end{align*}
for the state operator.
Assuming additionally $q \in \mathcal Y$, we may integrate by parts to obtain a strong formulation of the adjoint system. This yields the following optimality system.
\begin{thm}\label{prop:optimalityL2}
	Let $(f,u)$ be an optimal pair. The optimality condition corresponding to \eqref{OCMF} reads
	$$ \int_0^T d_uJ_2(\con_t)[h_t^\con] - \int_\Omega D_\con v(f_t,\con_t)[h_t^\con]\cdot\nabla q_t\,f_t \dd x  \dd t = 0 \quad \text{for all } h^\con \in C_0^\infty(\R^{dM}),$$
	where $q \in \mathcal{Y}$ satisfies the adjoint PDE given by
	\begin{equation}\label{eq:adjointL2}
	\partial_t q_t - \int_\Omega K_1(y-x) \cdot \nabla q_t(y)f_t(y)\dd y + v(f_t,\con_t) \cdot \nabla q_t = \sum_{i=1}^L \partial_i j(\langle g_1, \mu \rangle, \dots, \langle g_L, \mu \rangle)\, g_i 
	\end{equation}
	supplemented with the terminal condition $g_T = 0$.
\end{thm}
The adjoint equation \eqref{eq:adjointL2} derived via the $L^2$-approach clearly resembles \eqref{eq:compL2adjoint}.
\begin{rem}\label{rem:relmircomesoadjoint}
	As before, in the case \eqref{eq:structureV} the optimality conditions can be given explicity as
	\begin{gather*}
	\frac{d^2}{dt^2} \bar u^{\ell} = \frac{1}{\lambda} \sum_{\ell=1}^M\int_\Omega \nabla K_2(x-\bar u_t^\ell) \nabla q_t(x) f_t(x) \dd x,\\ \bar u_0^{\ell} = 0 = \frac{d}{dt} \bar u_T^{\ell}  \quad \text{for all } \ell=1,\dots,M.
	\end{gather*}
	A comparison with the optimality condition on the micro indicates a relation between $\nabla q$ and $\xi$ which will be further discussed  in the following.
\end{rem}

\subsection{Relations between the approaches}
In this section we discuss the relation between the adjoint derived w.r.t.\ the 2-Wasserstein distance and the gradient flow equation corresponding to the Hamiltonian approach (cf.\ \cite{FornasierSolombrino}). In order to define the probability measure containing forward and backward information we first recall the flow formulation of the state system
\begin{align}\label{eq:state_flow}
\frac{\dd}{\dd t} Q_t(x) &= v(Q_t\#\mu_0,\con_t)\circ Q_t(x), \qquad Q_0(x) =x,\quad \mu_0 = \text{law}(x).
\end{align}
Further, we introduce the adjoint flow $A_t$ corresponding to $\xi_t$, defined by $A_t = \xi_t\circ Q_t$. Its evolution equation is given by
\begin{align}\label{eq:adjoint_flow}
\begin{aligned}
	\frac{\dd}{\dd t}A_t(x) &= -\bigl(\nabla v(Q_t\#\mu_0,\con_t) \circ Q_t(x) \bigr) A_t(x) - \int_{\R^d} (\nabla K_1)(Q_t(y) - Q_t(x))A_t(y)\dd  \mu_0(y) \\ &\hspace{25em} - \delta_\mu J_1(Q_t\# \mu_0)
\end{aligned}
\end{align}
with terminal condition $A_T(x)= 0$.
\begin{rem}
	We would like to point out that  \eqref{eq:adjoint_flow} can also be derived directly from the state flow with the help of a Lagrangian-approach w.r.t.\ the $L^2$-scalar product. A change of coordinates from the Lagrangian to the Eulerian perspective leads to \eqref{eq:xiW}.
\end{rem}
Due to the strong dependence of the adjoint flow on the forward flow, one may understand \eqref{eq:state_flow} and \eqref{eq:adjoint_flow} as a coupled system of equations.
Let us consider the measure $\nu\in \calC([0,T], \calP_1(\R^d\times\R^d))$ defined by the push-forward of $\mu_0$ along the map $S_t(x)=(Q_t(x),A_t(x))$ for all $x\in\R^d$ and $t\in[0,T]$:
\[
 \iint_{\R^d\times\R^d} \varphi(x,r)\dd \nu_t(x,r) = \int_{\R^d} (\varphi\circ S_t)(x)\dd \mu_0(x)\qquad \text{for all\, $\varphi\in \calC_b(\R^d\times\R^d)$}.
\]
Notice that since
\[
 \iint_{\R^d\times\R^d} \varphi(x)\dd \nu_t(x,r) = \int_{\R^d} (\varphi\circ Q_t)(x)\dd \mu_0(x) = \int_{\R^d} \varphi(x)\dd \mu_t(x),
\]
the first marginal of $\nu_t$ corresponds to $\mu_t$. In particular,
\begin{gather*}
v(\mu_t,\con_t) = v(\nu_t,\con_t),\qquad \nabla v(\mu_t,\con_t) = \nabla v(\nu_t,\con_t),\qquad \delta_\mu J_1(\mu_t) = \delta_\mu J_1(\nu_t), \\
\int_{\R^d} (\nabla K_1)(Q_t(y) - x)A_t(y)\dd  \mu_0(y) = \iint_{\R^d\times\R^d} (\nabla K_1)(y - x)\eta\dd  \nu_t(y,\eta).
\end{gather*}
Furthermore, 
\[
\iint_{\R^d\times\R^d} \varphi(r)\dd \nu_T(x,r) = \int_{\R^d} (\varphi\circ A_T)(x)\dd \mu_0(x) = \varphi(0),
\]
i.e.\ $\nu_T(\R^d \times B) = \delta_0(B)$ for all $B\in\mathcal{B}(\R^d)$.

From the definition of $\nu$, it is not difficult to see that $\nu$ satisfies
\begin{equation}\label{eq:evolutionNu}
\partial_t \nu_t + \nabla_x \cdot \Big( \nabla_\xi \mathcal H(\nu_t,\con_t) \nu_t \Big) - \nabla_\xi \cdot \Big(\nabla_x \mathcal H(\nu_t,\con_t) \nu_t \Big) = 0,
\end{equation}
with mixed initial and terminal data given by
\begin{equation*}
 \nu_0 (B \times \R^d) = \mu_0(B), \qquad \nu_T(\R^d \times B) = \delta_0(B) \quad \text{for any\, $B\in \mathcal{B}(\R^d)$},
\end{equation*}
where the {\em Hamiltonian} (cf.~\cite{FornasierPP}) corresponding to \eqref{OCMF} is given by
\begin{align}\label{eq:hamiltonian}
\begin{aligned}
	\mathcal H(\nu,\con)(x,\xi) &= v(\nu,\con)(x)\cdot \xi + \iint_{\R^d\times\R^d} K_1(y-x) \cdot \eta\dd  \nu(y,\eta) \\
	&\hspace*{15em}-\sum_{i=1}^L (\partial_i j)(\langle g_1 ,  \nu \rangle, \dots, \langle g_L,  \nu \rangle) g_i(x).
\end{aligned}
\end{align}

On the other hand, \eqref{eq:evolutionNu} can also be derived from a mean-field Ansatz \cite{FornasierSolombrino}. Indeed, starting from the system of forward and adjoint ODEs, leads to the empirical measure $\nu^N$ defined as
\begin{equation}\label{eq:empiricalNu}
\nu_t^N(\dd x\dd \xi) = \frac{1}{N}\sum_{i=1}^N \delta_{(x_t^i,\xi_t^i)}(\dd x\dd\xi) 
\end{equation} which satisfies \eqref{eq:evolutionNu}. More details can be found in, e.g., \cite{diss,sheep1}. 

We conclude this section with a discussion of the relation of $\nu$ and the vector-valued adjoint variable $m$ defined by \eqref{eq:adjoint}. More precisely, we show that $m$ satisfying \eqref{eq:adjoint} can be characterized as first moment of $\nu$ with respect to $\xi$. We use the notation $\omega_t(\dd x) := \int_{\R^d} \xi \nu_t(\dd x,\dd\xi)$. Since by construction,
\begin{align*}
 \int_{\R^d} \varphi(x)\dd |\omega_t|(x) &\le \iint_{\R^d\times\R^d} \varphi(x) |\xi| d\nu_t(\dd x\dd\xi) = \int_{\R^d} (\varphi \circ Q_t)(x) |A_t(x)| \dd\mu_0(dx) \\
 &\le \|\varphi\|_{\sup} \|\xi_t\|_{\sup} \le \|\varphi\|_{\sup} \sup\nolimits_{t\in[0,T]}\|\xi_t\|_{\sup}\qquad\text{for all\, $\varphi\in\calC_b(\R^d)$},
\end{align*}
the measure $\omega_t$ is well-defined, and satisfies \eqref{eq:adjoint} with the terminal condition 

The above discussion yields the following result:
\begin{prop}
The adjoint corresponding to \eqref{OCMF} derived in the Wasserstein space $\calP_2(\R^d)$ solves \eqref{eq:xiW} and can be characterized as the first moment w.r.t.~$\xi$ of the probability measure $\nu$ corresponding to the Hamiltonian flow \eqref{eq:evolutionNu} of \eqref{OCMF} with Hamiltonian given by \eqref{eq:hamiltonian}.  
\end{prop}

The findings of this section are summarized in Figure~\ref{fig:flowChart}. On the ODE level the adjoints can be computed using the $L^2$-approach. Passing to the mean-field limit with the empirical measure \eqref{eq:empiricalNu} yields an evolution equation for a probability measure on the state and adjoint space \eqref{eq:evolutionNu}. The first $\xi$-moment of $\nu$ satisfies the same equation as the adjoint equation derived in the space of probability measures equipped with the 2-Wasserstein distance, i.e.\ \eqref{eq:adjoint}. The evolution of point masses following the characteristics of the mean-field adjoint equation equals the solution of the adjoint ODE with states initialized at the corresponding points. Moreover, we formally obtain a relation of the $L^2$-adjoint \eqref{eq:adjointL2} and \eqref{eq:xiW}, whenever $K_1$ and $K_2$ are gradients of potential fields. Indeed, taking the gradient of the evolution equation of $g$ yields \eqref{eq:xiW} for $\nabla g = \xi$ (see Remarks~\ref{rem:adjoint_potential} and \ref{rem:relmircomesoadjoint}).	

\begin{figure}[ht!]
 \includegraphics[width=0.85\textwidth]{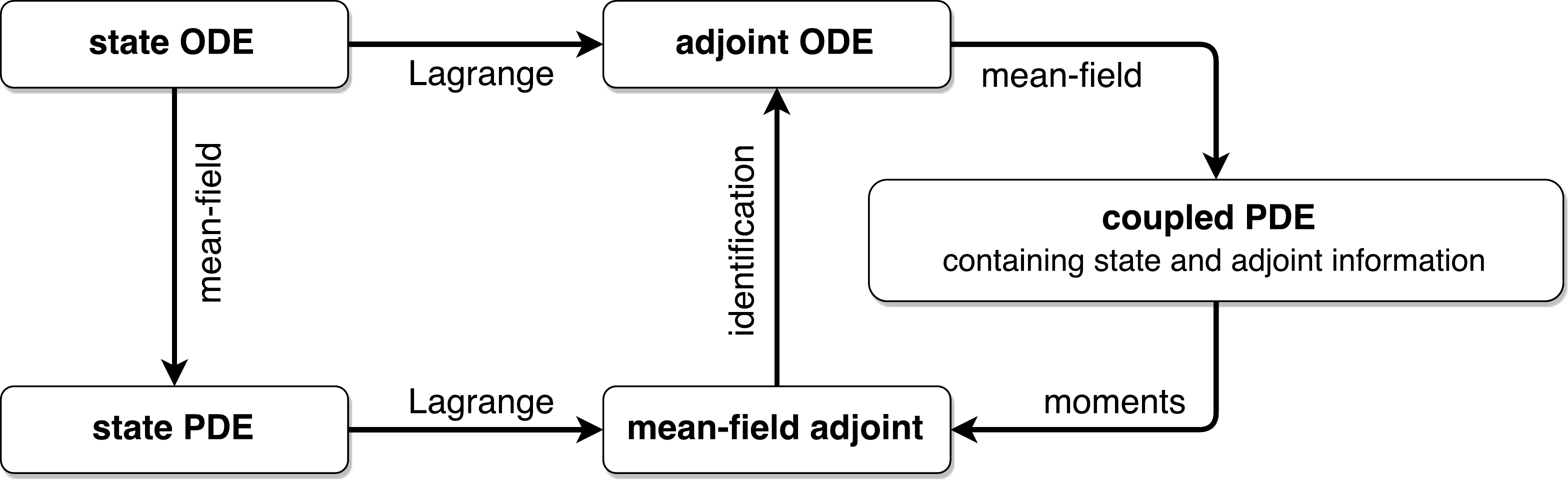}
 \caption{\small Flow chart showing the relations between the different adjoint approaches discussed in this section.}
 \label{fig:flowChart}
\end{figure}

\begin{rem}\label{justificationNumerics}
As the adjoint equation obtained using the calculus in the space of probability measures is vector-valued, it may be infeasible for numerical simulations for higher space dimensions. The link between the vector-valued adjoint and the $L^2$-adjoint discussed in this section can be seen as justification to use the $L^2$-adjoint for numerics. Indeed, in \cite{sheep1} this procedure leads to very convincing results.
\end{rem}

\section{Convergence Rate}\label{sec:ConvergenceRate}
\noindent
In this section we investigate the convergence of the microscopic optimal controls to the optimal control of the mean-field problem as $N \rightarrow \infty$. Our strategy for the proof is to use flows to pull the information back to the initial data. For simplicity we assume that $J_2(u)$ and $v$ have the structures given in Remark~\ref{rem:costfun} and \eqref{mfV}, respectively, in more detail:
$$ J_2(\con) = \frac{\lambda}{2}\left\| \frac{d\con}{dt} \right\|_{L^2((0,T),\R^{dM})}^2, \qquad v(\mu,\con) = -K_1 \ast \mu - \sum_{\ell=1}^M  K_2(x - u^\ell).$$
 For the initial data we assume convergence as $N \to \infty$. This can be realized by drawing samples from the initial measure $ \hat \mu$ for the particles (see Remark~\ref{rem:convergenceInitalmeasures}). 

\medskip

To summarize, the goal of this section is to prove
\begin{thm}\label{thm:convergenceRate}
Let the assumptions \ref{A1}-\ref{A6} hold and $J_2(\con)$ as above. Further, let $(\bar\x^N,\bar\con^N)$ and $(\bar{\mu},\bar\con)$ be optimal pairs for \eqref{OCN} and \eqref{OCMF} with initial data $\hat\x^N$, $\hat\mu$, respectively. Moreover, let the adjoint velocity for the pair $(\bar\mu,\bar\con)$ satisfy $\bar\xi\in \calC([0,T],\Lip_b(\R^d))$. Then there exists a constant $\gamma>0$ depending only on $T$, $K_1$, $K_2$, $J_1$ and Lipschitz bounds on $\bar{\xi}$, such that for $\lambda>\gamma$ it holds
\[
 \norm{\con^N - \bar\con}_{L^2((0,T), \mathbb{R}^{dM})}^2 \le \frac{\gamma}{\lambda-\gamma} W_2^2(\mu_{\hat x^N}^N,\hat\mu),
\]
where $\mu_{\hat x^N}^N$ denotes the empirical measures corresponding to the initial configurations $\hat\x^N$.
\end{thm}

\begin{rem}\label{growthRate_SSC}
Note, that we cannot expect that the solutions of the respective optimal control problems are unique. Hence, we need to ensure that our problem is convex enough, i.e.\ $\lambda$ is large enough. Essentially, we require here some kind of second-order sufficient condition or, equivalently, a quadratic growth condition near to the optimal state (see also \cite{casas2015}).
\end{rem}

\begin{rem}\label{rem:convergenceInitalmeasures}
Theorem~\ref{thm:convergenceRate} show that the convergence rate strongly depends on the convergence of the initial measures $W_2(\mu_{\hat x^N}^N,\hat\mu)\to 0$. Since $\mu$ is assumed to have compact support, we obtain a convergence of order $\sqrt{N}$ (cf.~\cite{convergenceInitial}), if one chooses $\hat\x^N$ as random variables with distribution $\hat\mu$.  
\end{rem}
\begin{rem}
	The proof of the convergence rate can be obtained as well in a slightly different setting, i.e.\ without fixing the initial positions of the controls. Indeed, for $$J_2(\con) := \frac{\lambda}{2 }\int_0^T \left| \frac{d\con_t}{dt}  \right|^2 + | \con_t - \con_0|^2 \dd t\quad \text{and} \quad \Uad = H^1(0,T; \R^{dM}), $$ one obtains a similar proof without using a Poincar\'e inequality.
\end{rem}

We begin with a simple result (without proof) on $v=v(\mu,\con)$ and $J_1$.

\begin{lem}\label{lem:v_condition}
	\begin{enumerate}[label=(\roman*)]
		\item Under assumption \ref{A5}, the mapping $v: \calP_2(\R^d) \times \R^{dM} \to \calC_b^2(\R^d)$ defined by \eqref{mfV} satisfies for any $\mu,\mu'\in\calP_2(\R^d)$ and $\con,\con'\in \R^{dM}$
		\[
		\|v(\mu,\con)-v(\mu',\con')\|_{\sup} +\|Dv(\mu,\con)-Dv(\mu',\con')\|_{\sup} \le C_v\Bigl(W_2(\mu,\mu') + |\con - \con'|\Bigr),
		\]
		for some constant $C_v$, independent of $(\mu,\con)$ and $(\mu',\con')$.
		\item If in addition to \ref{A3}, $j\in\calC^2(\R^K)$ and $g_l\in\calC^2(\R^d)$, $l=1,\ldots,K$ then $\delta_\mu J_1$ defined in \eqref{eq:deltaJ_1} satisfies
		\[
		 \|\delta_\mu J_1(\mu) - \delta_\mu J_1(\mu')\|_{\sup} \le C_{J_1}W_2(\mu,\mu')
		\]
		for some constant $C_{J_1}$, depending only on $J_1$, $\mathfrak{m}_2(\mu)$ and $\mathfrak{m}(\mu')$.
	\end{enumerate}
\end{lem}

\begin{rem}\label{rem:v_condition_uniform}
 Note that if $\sup_{t\in[0,T]} \{\mathfrak{m}_2(\mu_t) + \mathfrak{m}_2(\mu_t')\}<\infty$, then the time dependent constants $C_{J_1}(t)$ in Lemma~\ref{lem:v_condition} are uniformly bounded in $t$, i.e.\ $\sup_{t\in[0,T]} C_{J_1}(t)<\infty$.
\end{rem}

We now proceed with a stability estimate for the adjoint {\em velocities} ${\boldsymbol\xi}^N$ and $\xi$ corresponding to the equations \eqref{eq:adjointODE_explicit} and \eqref{eq:xiW}, respectively.

\begin{rem}\label{rem:adjointODE_concise}
	Equation \eqref{eq:adjointODE_explicit} can be written in the concise integral form
	\begin{align}\label{eq:adjointODE_concise}
	 \xi_t^{N,i} = -\int_t^T \Psi^N_i(\mu_s^N,\con_s)[\bar{\boldsymbol\xi}_s^N]\dd s,\qquad \xi_T^{N,i} = 0,\qquad i=1,\ldots,N,
	\end{align}
	where $\Psi^N_i$ is given by
	\[
	 \Psi_{i}^N(\mu^N,\con^N)[{\boldsymbol\xi}^N] = -\nabla v(\mu^N,\con^N)(x^{N,i})\,\xi^{N,i} - \frac{1}{N}\,\sum_{j=1}^N\nabla K_1(x^{N,j}-x^{N,i})\xi^{N,j} + \delta_\mu J_1(\mu^N)(x^{N,i}),
	\]
	in connection to the operator $\Psi$ defined in \eqref{eq:Psi}.
\end{rem}

\begin{lem}\label{lem:xi_stability}
	Let the assumptions \ref{A1}-\ref{A6} hold. Further, let $\x^N$ and $\mu$ be solutions to \eqref{eq:stateODE} and \eqref{eq:PDE} for given controls $\con^N$, $\con$ and initial data $\hat\x^N$, $\hat\mu$, respectively. If $\boldsymbol{\xi}^N$ satisfies \eqref{eq:adjointODE_explicit} for the pair $(\x^N,\con^N)$ and $\xi\in \calC([0,T],\Lip_b(\R^d))$ satisfies \eqref{eq:xiW} for the pair $(\mu,\con)$, then there exist positive constants $a$ and $b$, independent of $N\in\N$ such that
	\[
	 \sup_{t\in[0,T]}\frac{1}{N}\sum_{i=1}^N |\xi_{t}^{N,i} - \xi_{t}(x_{t}^{N,i})| \le be^{a T} \int_0^T \Bigl(W_2(\mu_{s}^N,\mu_{s}) + |\con_{s}^N - \con_{s}|\Bigr)\dd s.
	\]
\end{lem}
\begin{proof}
 	Denote by $\mu^N$ the empirical measure corresponding to the particles $\x^N$. We further denote $C_{v,J_1}:= C_v+\sup\nolimits_{t\in[0,T]} C_{J_1}(t)$ with $C_v$ and $C_{J_1}(t)$ given in Lemma~\ref{lem:v_condition} for each $t\in[0,T]$. Due to Remark~\ref{rem:v_condition_uniform}, $C_{v,J_1}<\infty$. From Remark~\ref{rem:adjointODE_concise}, we see that $\boldsymbol{\xi}^N$ satisfies \eqref{eq:adjointODE_concise}, and therefore,
 	\[
 	 \xi_t^{N,i} - \xi(x_t^{N,i}) = -\int_t^T \Bigl[ \Psi^N_i(\mu_s^N,\con_s)[{\boldsymbol\xi}_s^N] - \Psi(\mu_s,\con_s)[\xi_s](x_s^{N,i})\Bigr]\dd s = - \int_t^T \text{(I)} + \text{(II)} + \text{(III)}\; ds,
 	\]
 	where
 	\begin{align*}
		\text{(I)} &= -\nabla v(\mu_s^N,\con_s^N)(x_s^{N,i})\,\xi_s^{N,i} + \nabla v(\mu_s,\con_s)(x_s^{N,i})\,\xi_s(x_s^{N,i}), \\
		\text{(II)} &= -\frac{1}{N}\,\sum_{j=1}^N\nabla K_1(x_s^{N,j}-x_s^{N,i})\xi_s^{N,j} + \int \nabla K_1(y-x_s^{N,i})\xi_s(y)\dd \mu(y),\\
		\text{(III)} &= \delta_\mu J_1(\mu_s^N)(x_s^{N,i})- \delta_\mu J_1(\mu_s)(x_s^{N,i}).
 	\end{align*}
 	From Lemma~\ref{lem:v_condition}, we easily deduce that
 	\begin{align*}
		|\text{(I)}| &\le C_{v,J_1}\Bigl(W_2(\mu_s^N,\mu_s) + |\con_s^N - \con_s|\Bigr)|\xi_s^{N,i}| + \|\nabla v(\mu_s,\con_s)\|_{\sup}|\xi_s^{N,i}  - \xi_s(x_s^{N,i})|,\\
		|\text{(III)}| &\le C_{v,J_1} W_2(\mu_s^N,\mu_s).
 	\end{align*}
 	As for (II), we have
 	\begin{align*}
		|\text{(II)}| &\le \frac{1}{N}\,\sum_{j=1}^N\bigl|\nabla K_1(x_s^{N,j}-x_s^{N,i})\bigr|\bigl|\xi_s^{N,j}- \xi_s(x_s^{N,j})\bigr| \\
		&\hspace*{6em}+ \iint \bigl|\nabla K_1(y-x_s^{N,i})\xi_s(y)-\nabla K_1(y'-x_s^{N,i})\xi_s(y')\bigr|d\pi_s(y,y') \\
		&\le \|DK_1\|_{\sup}\frac{1}{N}\,\sum_{j=1}^N |\xi_s^{N,j}- \xi_s(x_s^{N,j})| \\
		&\hspace*{6em}+ \Bigl(\|D^2K_1\|_{\sup}\|\xi_s\|_{\sup} + \|DK_1\|_{\sup}\Lip(\xi_s)\Bigr) W_2(\mu_s^N,\mu_s),
 	\end{align*}
 	where $\pi_s$ is an optimal coupling between $\mu_s^N$ and $\mu_s$.
 	
 	Defining
 	\[
 	 Y_t^N := \frac{1}{N}\sum_{i=1}^N |\xi_t^{N,i} - \xi_t(x_t^{N,i})|,
 	\]
 	we find positive constants $a,b>0$, independent of $N$ (cf.~Remark~\ref{rem:adjointODE_N}) such that
 	\begin{align*}
 	Y_t^N \le a\int_t^T Y_s^N\dd s + b \int_t^T \Bigl(W_2(\mu_s^N,\mu_s) + |\con_s^N - \con_s|\Bigr)\dd s.
 	\end{align*}
 	An application of Gronwall's inequality gives
 	\[
 	 Y_{T-t}^N \le  be^{a t} \int_0^t \Bigl(W_2(\mu_{T-s}^N,\mu_{t_s}) + |\con_{T-s}^N - \con_{T-s}|\Bigr)\dd s.
 	\]
 	Taking the supremum over $t\in[0,T]$ yields the required estimate.
\end{proof}

\begin{rem}\label{rem:xi_stability_uniform}
	Putting Lemma~\ref{lem:xi_stability} and Lemma~\ref{lem:forwardDobrushin} together, we obtain the estimate
	\[
	 \sup_{t\in[0,T]}\frac{1}{N}\sum_{i=1}^N |\xi_{t}^{N,i} - \xi_{t}(x_{t}^{N,i})|^2 \le C_T\Bigl(W_2^2(\hat\mu^N,\hat\mu) + \|\con^N-\con\|_{L^2((0,T),\R^{dM})}^2\Bigr),
	\]
	for some positive constant $C_T$, independent of $N\in\N$.
\end{rem}

\begin{proof}[Proof of Theorem~\ref{thm:convergenceRate}]
	In the following, let $(\bar\x^N,\bar\con^N)$ and $(\bar{\mu},\bar\con)$ be optimal pairs for \eqref{OCN} and \eqref{OCMF} respectively. Further, let $\bar{\boldsymbol\xi}^N$ and $\bar\xi$ be adjoint velocities of the $N$-particle trajectories and mean-field limit corresponding to \eqref{eq:adjointODE_explicit} and \eqref{eq:xiW} respectively. We also denote by $\bar{\mu}^N$ the empirical measure corresponding to the particles $\bar\x^N$.
	
	Recall the optimality conditions for $\bar\con^N$ and $\bar\con$, given by \eqref{eq:optconditionODE} and \eqref{eq:optconditionPDE}, respectively. Taking their differences and using $\h^N=\bar\con^N-\bar\con$ as a test function, we arrive at
	\[
	\frac{\lambda}{2}\left\| \frac{d}{dt}\h^N \right\|_{L^2((0,T),\R^{dM})}^2 = \bigl(dJ_2(\bar\con^N)- dJ_2(\bar\con)\bigr)[\h^N] =  \sum_{\ell=1}^M \int_0^T  \h_t^{N,\ell}\cdot\textbf{f}_t^{\,N,\ell}\, dt,
	\]
	with
	\begin{align*}
	\textbf{f}_t^{\,N,\ell} &= \frac{1}{N}\sum_{i=1}^N (\nabla K_2)(\bar x^{N,i}_t - \bar\con_t^{N,\ell})\,\bar\xi_t^{N,i} - \int_{\R^d} (\nabla K_2)(x-\bar\con_t^{\ell})\,\bar \xi_t\dd \bar \mu_t 
	= \text{(I)} + \text{(II)},
	\end{align*}
	where
	\begin{align*}
	\text{(I)} &= \frac{1}{N}\sum_{i=1}^N (\nabla K_2)(\bar x^{N,i}_t - \bar\con_t^{N,\ell})\Bigl[\bar\xi_t^{N,i} - \bar\xi_t(\bar x_t^{N,i})\Bigr],\\
	\text{(II)} &= \int_{\R^d} (\nabla K_2)(x-\bar\con_t^{N,\ell})\,\bar \xi_t\dd \bar \mu_t^N - \int_{\R^d} (\nabla K_2)(x-\bar\con_t^{\ell})\,\bar \xi_t\dd \bar \mu_t.
	\end{align*}
	For (I), we obtain from Remark~\ref{rem:xi_stability_uniform}
	\[
	 |\text{(I)}| \le \sqrt{C_T}\|DK_2\|_{\sup}\Bigl(W_2(\hat\mu,\hat\mu') + \|\h^N\|_{L^2((0,T),\R^{dM})}\Bigr).
	\]
	As for (II), we obtain in a similar manner as in the proof of Lemma~\ref{lem:xi_stability}
	\begin{align*}
	 |\text{(II)}| \le \Bigl(\|D^2K_2\|_{\sup}\|\bar\xi_t\|_{\sup}|\h_t^{N,\ell}| + \|DK_2\|_{\sup}\Lip(\bar\xi_t)\Bigr)W_2(\bar\mu_t^N,\bar{\mu}_t).
	\end{align*}
	
	Altogether, we obtain a positive constant $c_0$, depending only on $T$, $K_1$, $K_2$, $j$, $g_l$, $l=1,\ldots,L$ and Lipschitz bound on $\bar{\xi}$ such that
	\begin{align*}
		\sum_{\ell=1}^M \int_0^T  \h_t^{N,\ell}\cdot\textbf{f}_t^{\,N,\ell}\, dt \le c_0\Bigl(W_2^2(\mu_0^N,\mu_0) + \|\h^N\|_{L^2((0,T),\R^{dM})}^2\Bigr).
	\end{align*}
	On the other hand, from the Poincar\'e inequality, we have a constant $c_P>0$ such that
	\[
	 \|\h^N\|_{L^2((0,T),\R^{dM})}^2 \le c_P \left\| \frac{d}{dt}\h^N \right\|_{L^2((0,T),\R^{dM})}^2,
	\]
	and consequently gives
	\[
	 (\lambda - 2 c_P c_0) \|\h^N\|_{L^2((0,T),\R^{dM})}^2 \le 2c_P c_0 W_2^2(\mu_0^N,\mu_0).
	\]
	For $\lambda > \gamma:=2 c_P c_0$, we may simply reformulate the inequality above and conclude the proof.
\end{proof}

\begin{rem}
	Note, that the same estimates in the proof of Theorem~\ref{thm:convergenceRate} may be used to provide uniqueness of minimizers to \eqref{OCN} and \eqref{OCMF}. See also Remark \ref{growthRate_SSC}.
\end{rem}

\appendix

\section{} \label{appex:forwardDobrushin}

\begin{proof}[Proof of Lemma~\ref{lem:forwardDobrushin}]
	Under the given assumptions, the solutions $\mu$ and $\mu'$ satisfy the continuity equations
	\[
	\partial_t\mu_t + \nabla\cdot(v(\mu_t,\con_t)\mu_t)=0,\qquad \partial_t\mu_t' + \nabla\cdot(v(\mu_t',\con_t')\mu_t')=0,\qquad\text{in distribution},
	\]
	with locally Lipschitz vector fields $v(\mu_t,\con_t)$ and $v(\mu_t',\con_t')$ for every $t\in[0,T]$ satisfying
	\[
	 \int_0^T \left(\int_{\R^d} |v(\mu_t,\con_t)|^2 d\mu_t^N + \int_{\R^d}|v(\mu_t',\con_t')|^2 d\mu_t'\right)dt<\infty.
	\]
	In this case, we can take the temporal derivative of $W_2^2(\mu_t^N,\mu_t)$ to obtain
	\begin{align*}
		\frac{1}{2}\frac{d}{dt} W_2^2(\mu_t,\mu_t') &= \iint_{\R^d\times\R^d} \langle v(\mu_t,\con_t)(x)-v(\mu_t',\con_t')(y),x-y\rangle\dd \pi_t(x,y) \\
		&\le \iint_{\R^d\times\R^d} \langle v(\mu_t,\con_t)(x) - v(\mu_t,\con_t)(y),x-y\rangle\dd \pi_t(x,y) \\
		&+ \iint_{\R^d\times\R^d} |v(\mu_t,\con_t)(y) -v(\mu_t',\con_t')(y)|\,|x-y|\dd \pi_t(x,y) =: I_1 + I_2,
	\end{align*}
	where $\pi_t$ is the optimal transference plan of $\mu_t$ and $\mu_t'$ for each $t\in[0,T]$. 
	
	For the first term, we easily deduce from \ref{A1} the following estimate
	\[
		I_1 \le C_l\iint_{\R^d\times\R^d} |x-y|^2 d\pi_t(x,y),
	\]
	As for the other term, we have, due to \ref{A2},
	\begin{align*}
		I_2 &\le C_v\Bigl(W_2(\mu_t,\mu_t') + \|\con_t-\con_t'\|_2\Bigr) \iint_{\R^d\times\R^d} |x-y|\dd \pi(x,y) \\
		&\le  \frac{C_v}{2}\Bigl(3W_2^2(\mu_t,\mu_t') + 2\|\con_t-\con_t'\|_2^2\Bigr),
	\end{align*}
	where the Young inequality was used in the last inequality. Altogether, we obtain
	\[
	 \frac{d}{dt} W_2^2(\mu_t,\mu_t') \le a W_2^2(\mu_t,\mu_t') + b \|\con_t-\con_t'\|_2^2,
	\]
	with time-independent constants $a,b>0$. Applying the Gronwall inequality on the quantity $e^{-a t}W_2^2(\mu_t,\mu_t')$, we finally obtain the required estimate.
\end{proof}

\section*{Acknowledgements} \noindent
CT acknowleges the discussions with J.A. Carrillo during her stay at Imperial College London, partially supported by the German Academic Exchange Service, where she studied the Wasserstein distance and related topics that influenced this manuscript. OT acknowledges support from NWO Vidi grant 016.Vidi.189.102, ``Dynamical-Variational Transport Costs and Application to Variational Evolutions".

\bibliographystyle{plain}
\bibliography{biblio}
\end{document}